\documentclass[a4paper,12pt, reqno]{amsart}
\usepackage{amssymb,amsthm,amsmath,bm,txfonts,graphicx,mathptmx}
\usepackage{array}
\usepackage{color}
\usepackage[all]{xy}

\makeatletter
  
  \@addtoreset{equation}{section}
\makeatother

\newtheorem{thm}{Theorem}[section]
\newtheorem{dfn}[thm]{Definition}
\newtheorem{df-th}[thm]{Definition-Theorem}
\newtheorem{lem}[thm]{Lemma}
\newtheorem{prop}[thm]{Proposition}
\newtheorem{ex}[thm]{Example}
\newtheorem{cor}[thm]{Corollary}
\newtheorem{cla}[thm]{Claim}

\theoremstyle{remark}
\newtheorem{rem}[thm]{Remark}

\newcommand{\rmI}{\mathrm{I}}

\newcommand{\calA}{\mathcal{A}}

\newcommand{\calC}{\mathcal{C}}
\newcommand{\calD}{\mathcal{D}}

\newcommand{\calI}{\mathcal{I}}

\newcommand{\calO}{\mathcal{O}}
\newcommand{\calP}{\mathcal{P}}

\newcommand{\calS}{\mathcal{S}}
\newcommand{\calT}{\mathcal{T}}

\newcommand{\fkm}{\mathfrak{m}}

\newcommand{\fkS}{\mathfrak{S}}

\newcommand{\NN}{\mathbb{N}}
\newcommand{\ZZ}{\mathbb{Z}}

\newcommand{\CC}{\mathbb{C}}
\newcommand{\FF}{\mathbb{F}}

\newcommand{\Spec}{\operatorname{Spec}}

\newcommand{\rad}{\operatorname{rad}}
\newcommand{\rank}{\operatorname{rank}}

\newcommand{\Hom}{\operatorname{Hom}}

\newcommand{\Image}{\operatorname{Im}}

\newcommand{\Ker}{\operatorname{Ker}}

\newcommand{\GL}{\operatorname{GL}}
\newcommand{\SL}{\operatorname{SL}}

\newcommand{\charac}{\operatorname{char}}
\newcommand{\CM}{\operatorname{CM}}

\newcommand{\Irr}{\operatorname{Irr}}
\newcommand{\twoheadlongrightarrow}{\relbar\joinrel\twoheadrightarrow}

\begin{document}

\title[Dual $F$-signature of CM modules over rational double points]{Dual $F$-signature of Cohen-Macaulay modules \\ over rational double points}
\author[YUSUKE NAKAJIMA]{YUSUKE NAKAJIMA}
\date{}


\address{Graduate School Of Mathematics, Nagoya University, Chikusa-Ku, Nagoya,
 464-8602 Japan} 
\email{m06022z@math.nagoya-u.ac.jp}
\maketitle

\begin{abstract}
The dual $F$-signature is a numerical invariant defined via the Frobenius morphism in positive characteristic. 
It is known that the dual $F$-signature characterizes some singularities. However, the value of the dual $F$-signature 
is not known except in only a few cases. In this paper, we determine the dual $F$-signature of Cohen-Macaulay modules over two-dimensional rational double points. 
The method for determining the dual $F$-signature is also valid for determining the Hilbert-Kunz multiplicity. We discuss it in appendix. 
\end{abstract}

\tableofcontents

\section{Introduction}

Throughout this paper, we suppose that $k$ is an algebraically closed field of prime characteristic $p>0$.
Let $R$ be a Noetherian ring of prime characteristic $p>0$, then 
we can define the Frobenius morphism $F:R\rightarrow R\;(r\mapsto r^p)$. 
For $e\in\mathbb{N}$, we also define the $e$-times iterated Frobenius morphism $F^e:R\rightarrow R\;(r\mapsto r^{p^e})$. 
For any $R$-module $M$, we define the $R$-module ${}^eM$ via $F^e$ as follows. 
That is, ${}^eM$ is just $M$ as an abelian group, and its $R$-module structure is defined by $r\cdot m\coloneqq F^e(r)m=r^{p^e}m$ $\;(r\in R,\;m\in M)$.
We say $R$ is $F$-finite if ${}^1R$ $($and hence every ${}^eR)$ is a finitely generated $R$-module. 
For example, if $R$ is an essentially of finite type over a perfect field or 
complete Noetherian local ring with a perfect residue field $k$, then $R$ is $F$-finite. 
In this paper, we only discuss such rings, thus the $F$-finiteness is always satisfied. 

\bigskip

In positive characteristic commutative algebra, we understand the properties of $R$ through the structure of ${}^eM$. 
For this purpose, several numerical invariants are defined. 
Firstly, we introduce the notion of $F$-signature defined by C.~Huneke and G.~Leuschke. 

\begin{dfn}[\cite{HL}]
\label{HLFsig}
Let $(R,\fkm,k)$ be a $d$-dimensional reduced $F$-finite Noetherian local ring with $\charac\,R=p>0$.
For each $e\in\mathbb{N}$, we decompose ${}^eR$ as follows
\[
{}^eR\cong R^{\oplus a_e}\oplus M_e,
\]
where $M_e$ has no free direct summands. We call $a_e$ the $e$-th $F$-splitting number of $R$.
Then, we call the limit 
\[
 s(R)\coloneqq\lim_{e\rightarrow\infty}\frac{a_e}{p^{ed}},
\]
the $F$-signature of $R$.
\end{dfn}

The existence of the $F$-signature of $R$ was shown by K. Tucker \cite{Tuc}. By Kunz's theorem, $R$ is regular if and only if  
${}^eR$ is a free $R$-module of rank $p^{ed}$ \cite{Kun}. 
Thus, roughly speaking, the $F$-signature $s(R)$ measures the deviation from regularity. 
The next theorem confirms this intuition. 

\begin{thm}[\cite{HL}, \cite{Yao2}, \cite{AL}]
Let $(R,\fkm,k)$ be a $d$-dimensional reduced $F$-finite Noetherian local ring with $\charac\,R=p>0$. 
Then we have the following. 
\begin{itemize}
 \item [(1)] $R$ is regular if and only if $s(R)=1$, 
 \item [(2)] $R$ is strongly $F$-regular if and only if $s(R)>0$.
\end{itemize}
\end{thm}

   \medskip

This notion is extended for a finitely generated $R$-module as follows.
\begin{dfn}[\cite{San}]
\label{SanFsig}
Let $(R,\fkm,k)$ be a $d$-dimensional reduced $F$-finite Noetherian local ring with $\charac\,R=p>0$.
For a finitely generated $R$-module $M$ and $e\in\mathbb{N}$, set
\[
 b_e(M)\coloneqq\operatorname{max}\{n\;|\;\exists\varphi:{}^eM\twoheadrightarrow M^{\oplus n}\}, 
\]
and call it the $e$-th $F$-surjective number of $M$.

Then we call the limit
\[
 s(M)\coloneqq\lim_{e\rightarrow\infty}\frac{b_e(M)}{p^{ed}}
\]
the dual $F$-signature of $M$ if it exists.
\end{dfn}

\begin{rem}
Since the morphism ${}^eR\twoheadrightarrow R^{\oplus b_e(R)}$ splits, 
if $M$ is isomorphic to the basering $R$, then the dual $F$-signature of $R$ in sense of Definition\,\ref{SanFsig} 
coincides with the $F$-signature of $R$. Thus, we use the same notation unless it causes confusion. 
\end{rem}

Just like the $F$-signature, the dual $F$-signature also characterizes some singularities.

\begin{thm}[\cite{San}]
Let $(R,\fkm,k)$ be a $d$-dimensional reduced $F$-finite Cohen-Macaulay local ring with $\charac\,R=p>0$. Then we have
\begin{itemize}
 \item [(1)] $R$ is $F$-rational if and only if $s(\omega_R)>0$, 
 \item [(2)] $s(R)\le s(\omega_R)$,
 \item [(3)] $s(R)= s(\omega_R)$ if and only if $R$ is Gorenstein.\end{itemize}
\end{thm}

In this way, the value of $s(R)$ and $s(\omega_R)$ characterize some singularities. Now we have some questions. 
Let $M$ be a finitely generated $R$-module which may not be $R$ or $\omega_R$. Then 
\begin{itemize}
 \item [$\cdot$] Does the value of $s(M)$ contain any information about singularities ?
 \item [$\cdot$] What does the explicit value of $s(M)$ mean ?
 \item [$\cdot$] Is there any connection between $s(M)$ and other numerical invariants ? 
\end{itemize}

However, it is difficult to try these questions for now, 
because the value of the dual $F$-signature is not known and we don't have an effective method for determining it except in only a few cases. 
For example, the case of two-dimensional Veronese subrings is studied in \cite[Example\,3.17]{San} and the author determined the dual $F$-signature for a certain 
class of maximal Cohen-Macaulay (= MCM) modules over cyclic quotient surface singularities in \cite{Nak}. 
As far as the author knows, there is no other example. 
Thus, in this paper, we investigate the dual $F$-signature for MCM modules over two-dimensional rational double points
(or Du Val singularities, Kleinian singularities, ADE singularities in the literature). 
Since we already know the case of $A_n$ $($see \cite{Nak}$)$, we discuss the type $D_n, E_6, E_7$ and $E_8$ in this paper.

\bigskip

The structure of this paper is as follows.
In order to determine the dual $F$-signature, we have to understand the following topics: 
\begin{itemize}
\item[(1)] The structure of ${}^eM$, namely 
  \begin{itemize} 
   \item What kind of MCM appears in ${}^eM$ as a direct summand?
   \item The asymptotic behavior of ${}^eM$ on the order of $p^{2e}$.
  \end{itemize}
\item[(2)] How do we construct a surjection ${}^eM\twoheadrightarrow M^{\oplus b_e}$? 
\end{itemize}
To show the former one, we need the notion of generalized $F$-signature. So we review it in Section\,\ref{Gene_Fsig}. 
After that we will use the notion of the Auslander-Reiten quiver to solve the latter problem. 
Thus, we give a brief summary of Auslander-Reiten theory in Section\,\ref{tau_cat}. 
Even though, the main body of this paper is devoted to observing two-dimensional rational double points 
(i.e. invariant subrings under the action of finite subgroups of $\SL(2,k)$\;), 
the arguments in Section\,\ref{Gene_Fsig} and \ref{tau_cat} are also valid for a more general situation. 
Thus, we discuss finite subgroups of the general linear group over $k$. 
In Section\,\ref{Sec_keylemma}, we prepare a technical lemma to construct a surjection ${}^eM\twoheadrightarrow M^{\oplus b_e}$ efficiently. 
In Section\,\ref{DFsig_RDP}, we actually determine the value of the dual $F$-signature of each MCM module for all the ADE cases. 
Since the strategy for determining the dual $F$-signature is almost the same, we will give a concrete explanation only for the case of $D_5$. 
For the other cases, we only mention an outline. 
In Section\,\ref{summary}, we collect the values of the dual $F$-signature determined in the previous section. 

Since the methods for determining the dual $F$-signature is also valid for determining the Hilbert-Kunz multiplicity, we will discuss it in appendix. 

\subsection*{Conventions}
As we noted in the beginning of this paper, we assume that $k$ is an algebraically closed field of $\charac\,k=p>0$. 
Throughout this paper, when we discuss a composition of morphisms $fg$, it means $f$ is followed by $g$, that is, $fg=g{\footnotesize\circ}f$. 
Similarly, for quivers an arrow $ab$ means $a$ is followed by $b$.  

\section{Generalized $F$-signature of invariant subrings}
\label{Gene_Fsig}

Let $G$ be a finite subgroup of $\GL(d,k)$ which contains no pseudo-reflections except the identity 
and assume that the order of $G$ is coprime to $p=\mathrm{char}\;k$.
We denote the invariant subring of a power series ring $S\coloneqq k[[x_1,\cdots,x_d]]$ under the action of $G$ by $R\coloneqq S^G$.  
For determining the dual $F$-signature of a finitely generated $R$-module $M$, we have to understand the structure of ${}^eM$. 
For example, we would like to know the direct sum decomposition of ${}^eM$,  
asymptotic behavior of the multiplicities of each direct summands on the order of $p^{ed}$.  
For this purpose, we review the results of generalized $F$-signature of invariant subrings due to \cite{HN}. 
Although we are interested in the two-dimensional case, the results in this section hold for an arbitrary dimension. 
So we will state for the $d$-dimensional case. 

   \bigskip

For understanding the structure of ${}^eM$, we introduce the notion of finite $F$-representation type defined 
by K.~Smith and M.~Van den Bergh \cite{SVdB} as follows.

\begin{dfn}[\cite{SVdB}] 
We say $R$ has finite $F$-representation type $($= FFRT$)$ by $\calS$ if 
there is a finite set $\calS$ of isomorphism classes of indecomposable finitely generated $R$-modules,
such that for any $e\in\mathbb{N}$, the $R$-module ${}^eR$ is isomorphic to a finite direct sum of elements in $\calS$. 
\end{dfn}

For example, since a power series ring $S$ is regular, ${}^eS$ is isomorphic to $S^{\oplus p^{ed}}$. 
Thus, $S$ has FFRT by $\{S\}$ and it is known that FFRT is inherited by a direct summand \cite[Proposition\,3.1.4]{SVdB}. 
So the invariant subring $R$ also has FFRT. More precisely, we have the following proposition.

\begin{prop}(\cite[Proposition\,3.2.1]{SVdB}) 
Let $V_0=k,V_1,\cdots,V_n$ be the full set of non-isomorphic irreducible representations of $G$.
We set $M_t\coloneqq(S\otimes_kV_t)^G\;\;(t=0,1,\cdots,n)$.  
Then we see that $R$ has finite $F$-representation type by the finite set $\{M_0\cong R,M_1,\cdots,M_n\}$.
\end{prop}

From this proposition, we can describe ${}^eR$ as follows.
\begin{equation}
\label{iso_p^ed}
{}^eR\cong R^{\oplus c_{0,e}}\oplus M_1^{\oplus c_{1,e}}\oplus\cdots\oplus M_n^{\oplus c_{n,e}}.
\end{equation}

\begin{rem}
Under the assumption $G$ contains no pseudo-reflections except the identity, 
we can see that each $M_t$ is an indecomposable maximal Cohen-Macaulay (=MCM) $R$-module and $M_s\not\cong M_t\; (s\neq t)$. 
Moreover, the multiplicities $c_{t,e}$ are determined uniquely in that case. For more details, we refer the reader to \cite[Section\,2]{HN}.
\end{rem}

Since the invariant subring $R$ has FFRT, the limit 
$\displaystyle\lim_{e\rightarrow\infty}\frac{c_{t,e}}{p^{de}}\;(t=0,1,\cdots,n)$ exists \cite{SVdB}, \cite{Yao}. 
We denote this limit by $s(R,M_t)\coloneqq\displaystyle\lim_{e\rightarrow\infty}\frac{c_{t,e}}{p^{de}}$ and call it the (generalized) $F$-signature of $M_t$. 
The value of $s(R,M_t)$ is known as follows.

\begin{thm}(\cite[Theorem\,3.4]{HN}) 
\label{gene-Fsig}
For $t=0,1,\cdots,n$, we have 
\[
s(R,M_t)=\frac{\rank_RM_t}{|G|}
\]
\end{thm}

\begin{rem}
In the case of $t=0$, we have $s(R,R)=s(R)$ and the above result is also due to \cite[Theorem\,4.2]{WY2}.
Moreover, a similar result holds for a finite subgroup scheme of $\mathrm{SL_2}$ \cite[Lemma\,4.10]{HS}. 
\end{rem}

As a corollary, we also have the next statement .

\begin{cor}(\cite[Corollary\,3.10]{HN}) 
\label{gene-Fsig-cor}
Suppose an MCM $R$-module ${}^eM_t$ decomposes as follows.
\[
{}^eM_t\cong R^{\oplus d^t_{0,e}}\oplus M_1^{\oplus d^t_{1,e}}\oplus\cdots\oplus M_n^{\oplus d^t_{n,e}}
\]
Then, for all $s,t=0,\cdots,n$, we have 
\[
s(M_t,M_s)\coloneqq\lim_{e\rightarrow\infty}\frac{d^t_{s,e}}{p^{de}}=(\rank_RM_t)\cdot s(R,M_s)=\frac{(\rank_RM_t)\cdot(\rank_RM_s)}{|G|}.
\]
\end{cor}

In the rest of this paper, we suppose that $d=2$, that is, $R$ is the invariant subring of $S=k[[x,y]]$ under the action of 
a finite subgroup $G\subset\GL(2,k)$ which contains no pseudo-reflections except the identity. 

\begin{rem}
In the two-dimensional case, it is well known that $R$ is of finite CM representation type, that is, it has only finitely many 
non-isomorphic indecomposable MCM $R$-modules $\{R,M_1,\cdots,M_n \}$. 
As Corollary\,\ref{gene-Fsig-cor} shows, every indecomposable MCM $R$-modules appear in ${}^eM_t$ as a direct summand for sufficiently large $e\gg 0$. 
Therefore, the additive closure $\mathrm{add}_R({}^eM_t)$ coincides with the category of MCM $R$-modules $\CM(R)$. 
So we can apply several results in so-called Auslander-Reiten theory to $\mathrm{add}_R({}^eM_t)$. 
We discuss it in the next section.
\end{rem}

\section{Auslander-Reiten theory and counting argument}
\label{tau_cat}

\subsection{Auslander-Reiten quiver}
In order to construct a surjection ${}^eM_t\twoheadrightarrow M_t^{\oplus b_e}$, we will use the Auslander-Reiten (=AR) quiver. 
By using it, we visualize relations among MCM $R$-modules and construct a surjection efficiently. 
For defining it, we need some notions so-called the AR sequence $($or almost split sequence in the literature$)$ and irreducible morphisms. 
So we review some results of Auslander-Reiten theory in this subsection.
For details, see some textbooks (e.g. \cite{LW},\;\cite{Yo}). 

\begin{dfn}[Auslander-Reiten sequence]
Let $(A,\fkm,k)$ be a Henselian CM local ring and $M,N$ be indecomposable MCM $A$-modules.
We say that a non-split short exact sequence
\[
0\rightarrow N\overset{f}{\rightarrow}L\overset{g}{\rightarrow}M\rightarrow 0
\]
is AR sequence ending in $M$ $($or starting at $N)$ if for all MCM modules $X$ and for any morphism $\varphi:X\rightarrow M$ 
which is not a split surjection there is a morphism $\phi:X\rightarrow L$ such that $\varphi=\phi g$.
\end{dfn}

If the AR sequence exists, it is unique up to isomorphism. 
It is known that there exists the AR sequence ending in $M_t$ for any indecomposable MCM $R$-modules $M_t=(S\otimes_kV_t)^G\;(t\neq 0)$ as follows.
In our situation, the AR sequence constructed by the Koszul complex over $S$ and a natural representation of $G$ as follows $($for detail, see \cite[Chapter 10]{Yo} $)$.

\medskip

In the case of $t\neq 0$, the AR sequence ending in $M_t$ is 
\[
0\longrightarrow (S\otimes_k(\wedge^2V\otimes_kV_t))^G\longrightarrow (S\otimes_k(V\otimes_kV_t))^G\longrightarrow M_t=(S\otimes_kV_t)^G\longrightarrow 0, 
\]
where $V$ is a natural representation of $G$.

In the case of $t=0$, there exists the following sequence
\[
0\longrightarrow\omega_R=(S\otimes_k\wedge^2V)^G\longrightarrow(S\otimes_kV)^G\longrightarrow R=S^G\longrightarrow k\longrightarrow 0.
\]
This sequence is called the fundamental sequence of $R$.

We call the left term of these sequences AR translation and denote by $\tau(M_t)$. 
On the other hand, we denote the right term of the AR sequence starting at $M_t$ by $\tau^{-1}M_t$. 
It is known that $\tau(M_t)\cong(M_t\otimes_R\omega_R)^{**}$, where $(-)^*=\Hom_R(-,R)$ stands for the $R$-dual functor \cite{Aus1}. 
Sometimes we denote the middle term of the AR sequence ending in $M_t$ by $E_{M_t}$ for $t=1,\cdots,n$.

   \medskip
Next, we introduce the notion of irreducible morphism. 

\begin{dfn}[Irreducible morphism]
Suppose $M$ and $N$ are MCM $R$-modules. We decompose $M$ and $N$ into indecomposable modules as 
$M=\oplus_iM_i$,\;$N=\oplus_jN_j$. Also, we decompose $\psi\in \Hom_R(M,N)$ along the above decomposition as 
$\psi=(\psi_{ij}:M_i\rightarrow N_j)_{ij}$. 
Then we define submodule $\rad_R(M,N)\subset\Hom_R(M,N)$ as
\[
 \psi\in\rad_R(M,N)\Longleftrightarrow \text{no $\psi_{ij}$ is an isomorphism}.
\]
Furthermore, we define submodule $\rad^2_R(M,N)\subset\Hom_R(M,N)$. The submodule $\rad^2_R(M,N)$ consists of 
morphisms $\psi:M\rightarrow N$ such that $\psi$ decomposes as $\psi=fg$,
{\scriptsize
\[\xymatrix@C=8pt@R=8pt{
   M\ar[rr]^{\psi}\ar[dr]_f&&N\\
   &X\ar[ur]_g& \\
}\]}
where $X$ is an MCM $R$-module,\;$f\in\rad_R(M,X)$,\;$g\in\rad_R(X,N)$.
We say that a morphism $\psi:M\rightarrow N$ is irreducible if $\psi\in\rad_R(M,N)\setminus\rad^2_R(M,N)$. 
In this setting, we define the $k$-vector space $\Irr_R(M,N)$ as 
\[
 \Irr_R(M,N)\coloneqq\rad_R(M,N)\big/\rad^2_R(M,N). 
\]
\end{dfn}

   \medskip
We are now ready to define the AR quiver.

\begin{dfn}[Auslander-Reiten quiver]
The AR quiver of $R$ is an oriented graph whose vertices are indecomposable MCM $R$-modules $\{R,M_1,\cdots,M_n\}$ and 
draw $\dim_k\Irr_R(M_s,M_t)$ arrows from $M_s$ to $M_t$ \;$(s,t=0,1,\cdots,n)$. 
\end{dfn}

Note that $\dim_k\Irr_R(M_s,M_t)$ is equal to the multiplicity of $M_s$ in the indecomposable decomposition of $E_{M_t}$. 
So we describe the AR quiver from the structure of AR sequences. 
More fortunately, the AR quiver of $R$ coincides with the McKay quiver of $G$ by \cite{Aus1}, 
so we can describe it from representations of $G$ $($for the definition of McKay quiver, refer to \cite[$(10.3)$]{Yo}$)$. 
Note that finite subgroups of $\GL(2,k)$ which contain no pseudo-reflections except the identity are classified in \cite{Bri} 
and their McKay quiver (equivalently AR quiver) are described in \cite{AR1}. 
Moreover, if $G$ is a finite subgroup of $\SL(2,k)$, then the associated quivers take the form of extended Dynkin diagrams 
$($see also the beginning of Section\,\ref{DFsig_RDP}$)$.

\subsection{Counting argument of Auslander-Reiten quiver}
\label{count_AR}

From Nakayama's lemma, when we discuss the surjectivity of ${}^eM_t\rightarrow M_t^{\oplus b}$, 
we may consider an MCM module $M_t$ as a vector space after tensoring by the residue field $k$. 
Thus, we investigate a basis of $M_t/\fkm M_t$, equivalently a set of minimal generators of $M_t$.

\medskip 

Let $M$ be a non-free indecomposable MCM $R$-module. 
The number of minimal generator $\mu_R(M)$ is equal to $\dim_kM/\fkm M$ and 

\begin{center}
\begin{tabular}{rcl}
$M$&$\cong$&$\Hom_R(R,M)$ \\
$\cup$&& \\
$\fkm M$&$\cong$&$\{R\overset{\text{non split}}{\rightarrow}R^{\oplus m}\rightarrow M\}$
\end{tabular}
\end{center}
for some $m\in \NN$. 
From this observation, we identify a minimal generator of $M$ with a morphism from $R$ to $M$ which doesn't 
factor through free modules except the starting point. 
We spend the rest of this subsection for describing such morphisms. 

\bigskip

Let $\CM(R)$ be the category of maximal CM $R$-modules. In order to find such morphisms, we define the stable category $\underline{\CM}(R)$ as follows. 
The objects of $\underline{\CM}(R)$ are same as those of $\CM(R)$ and the morphism set is given by 
\[
\underline{\Hom}_R(X,Y)\coloneqq\Hom_R(X,Y)\big/\calP(X,Y), \quad X,Y\in \CM(R)
\]
where $\calP(X,Y)$ is the submodule of $\Hom_R(X,Y)$ consisting of morphisms which factor through a free $R$-module.

\medskip

Assume that $R$ is not isomorphic to $\omega_R (\cong \tau R)$, that is, $R$ is not Gorenstein. 
Let 
\begin{equation}
\label{AR_M}
0\longrightarrow R\overset{g}{\longrightarrow}E\overset{f}{\longrightarrow}\tau^{-1}R\longrightarrow 0
\end{equation}
be the AR sequence ending in $\tau^{-1}R$. For the morphism of functor category 
\[
\underline{\Hom}_R(\tau^{-1}R,-)\overset{f\cdot -}{\longrightarrow}\underline{\Hom}_R(E,-) ,
\]
we define the covariant additive functor $\FF:\underline{\CM}(R)\rightarrow \calA$ as the cokernel of $(f\cdot -)$ 
\[
\underline{\Hom}_R(\tau^{-1}R,-)\overset{f\cdot -}{\longrightarrow}\underline{\Hom}_R(E,-)\longrightarrow \FF\longrightarrow 0 ,
\]
where $\calA$ is the category of abelian groups. It is easy to see $\Ker(f\cdot -)=0$. 
By properties of the AR sequence $($\ref{AR_M}$)$, any morphism $R\rightarrow M$ factors through $E$ and $gf=0$. 
Thus, on the short exact sequence  
\[
\underline{\Hom}_R(\tau^{-1}R,M)\overset{f\cdot -}{\longrightarrow}\underline{\Hom}_R(E,M)\longrightarrow \FF(M)\longrightarrow 0 ,
\]
the composition morphisms of $R\overset{g}{\rightarrow}E$ and non-zero elements of $\FF(M)$ are exactly what we wanted. 

\begin{rem}
In the case when $R$ is Gorenstein, we use the fundamental sequence 
\[
0\longrightarrow R\longrightarrow E \overset{f}{\longrightarrow}R\longrightarrow k\longrightarrow 0 
\]
instead of the AR sequence $($\ref{AR_M}$)$, and we obtain $\FF(M)\cong\underline{\Hom}_R(E,M)$ by similar arguments.
\end{rem}
In order to find non-zero elements of $\FF(M)$, we compute $\dim_k\FF(M)=\dim_k\underline{\Hom}_R(E,M)-\dim_k\underline{\Hom}_R(\tau^{-1}R,M)$. 
More precisely, we will find a $k$-basis of $\FF(M)$.  
For this purpose, the counting arguments of AR quiver plays a crucial role. 
This method first appeared in the work of Gabriel \cite{Gab} and it was also used for classifying special CM modules over quotient surface singularities \cite{IW}. 
For more details about the counting arguments of AR quiver, see e.g. \cite{Gab}, \cite{Iya}, \cite{IW}, \cite{Wem}. 
For simplicity, we give a brief review of this kind of arguments in the form of algorithm as follows $($cf. \cite[Section\,4]{Wem}$)$. 

\bigskip 

\begin{itemize}
\item [1.] In the AR quiver $Q$, we write a $1$ $($resp. $-1)$ at the position corresponding to $E$ $($resp. $\tau^{-1}R)$. 
For every MCM $R$-module $N$, we define the following number  
\begin{equation*}
\nu^{(0)}_N\coloneqq\lambda^{(0)}_N\coloneqq
\begin{cases}
1&\text{if $N=E$} \\
-1&\text{if $N=\tau^{-1}R$} \\
0&\text{otherwise}.
\end{cases}
\end{equation*}
\item [2.] Next, we consider all arrows out of $E$ in $Q$ and call the head of these arrows the first-step vertices of $E$. 
For every MCM $R$-module $N$, we set 
\begin{equation*}
\lambda^{(1)}_N\coloneqq
\begin{cases}
1+\nu^{(0)}_N&\text{if $N$ is a first-step vertex} \\
0&\text{otherwise}.
\end{cases}
\end{equation*}
Then we define 
\begin{equation*}
\nu^{(1)}_N\coloneqq
\begin{cases}
0&\text{if $N=R$} \\
\lambda^{(1)}_N&\text{otherwise}.
\end{cases}
\end{equation*}
For every first-step vertex $N_1$, we put the number $\lambda^{(1)}_{N_1}$ on the corresponding vertex. 
\item [3.] We consider all arrows out of the first-step vertices and call the head of these arrows the second-step vertices. 
For every MCM $R$-module $N$, we set 
\begin{equation*}
\lambda^{(2)}_N\coloneqq
\begin{cases}
-\nu^{(0)}_{\tau(N)}+\displaystyle\sum_{L_1\rightarrow N}\nu^{(1)}_{L_1}&\text{if $N$ is a second-step vertex} \\
0&\text{otherwise}.
\end{cases}
\end{equation*}
where $L_1$ runs over all first-step vertices. Then we define 
\begin{equation*}
\nu^{(2)}_N\coloneqq
\begin{cases}
0&\text{if $N=R$} \\
\lambda^{(2)}_N&\text{otherwise}.
\end{cases}
\end{equation*}
For every second-step vertex $N_2$, we write the corresponding number $\lambda^{(2)}_{N_2}$. 
\item [4.] Then we consider all arrows out of the second-step vertices. We call the head of these arrows the third-step vertices. 
For every MCM $R$-module $N$, we set 
\begin{equation*}
\lambda^{(3)}_N\coloneqq
\begin{cases}
-\nu^{(1)}_{\tau(N)}+\displaystyle\sum_{L_2\rightarrow N}\nu^{(2)}_{L_2}&\text{if $N$ is a third-step vertex} \\
0&\text{otherwise}.
\end{cases}
\end{equation*}
where $L_2$ runs over all second-step vertices. We set  
\begin{equation*}
\nu^{(3)}_N\coloneqq
\begin{cases}
0&\text{if $N=R$} \\
\lambda^{(3)}_N&\text{otherwise}.
\end{cases}
\end{equation*}
For every third-step vertex $N_3$, we write the corresponding number $\lambda^{(3)}_{N_3}$. 
\item [5.] Continuing with this process, we record the number $\lambda^{(i)}_N$ on each vertex $N$. 
Since $R$ is of finite CM representation type, we have $\lambda^{(i)}_N=0$ for some $i\in\NN$ sooner or later. Thus, we will stop there. 

The number $\lambda^{(i)}_N$ means that there are $\lambda^{(i)}_N$ non-zero morphisms in $\FF(N)$ for each corresponding vertex $N$, 
and such morphisms consist a $k$-basis of $\FF(N)$. Note that we have $\dim_k\FF(N)=\sum_{i\ge 0}\lambda^{(i)}_N$. 
\end{itemize}

\begin{ex}
\label{D_52_group}
Let $G$ be the following finite group 
\[
G\coloneqq \langle\;
    \begin{pmatrix} i&0 \\
                    0&-i 
    \end{pmatrix},\;
    \begin{pmatrix} 0&i \\
                    i&0 
    \end{pmatrix},\;
    \begin{pmatrix} \zeta_6&0 \\
                    0&\zeta_6 
    \end{pmatrix}           \;\rangle \subset\GL(2,k),
\]
where $\zeta_6$ is a primitive $6$-th root of unity. 
This group is isomorphic to $\calD_2\times Z_3$ where $Z_3$ is generated by the scalar matrix $\mathrm{diag}(\zeta_3,\zeta_3)$ and 
$\calD_2$ is the binary dihedral group of order $8$ $($see also the beginning of Section\,\ref{DFsig_RDP}$)$. 
Note that this group is denoted by $D_{5,2}$ in \cite{Rie}. 
Then we have finitely many irreducible representations 
\[
V_{i,j}\coloneqq V_i\otimes W_j \;\;(i=0,1,\cdots,4,\;j=0,1,2)
\]
where $W_j$ is a irreducible representation of $Z_3$ which represents $\mathrm{diag}(\zeta_3,\zeta_3)\mapsto\zeta_3^j$ and
$V_i$ is a that of $\calD_2$ associated to the extended Dynkin diagram
$\begin{array}{c}\tiny{\xymatrix@C=5pt@R=0.5pt{0\ar@{-}[rd]&&3\\ &2\ar@{-}[ru]\ar@{-}[rd]& \\1\ar@{-}[ru]&&4}}\end{array}$ 
and set the indecomposable MCM module $M_{i,j}\coloneqq (S\otimes_kV_{i,j})^G$. 
The AR quiver of $k[[x,y]]^G$ is the following $($for simplicity we only describe subscripts as vertices$)$;

{\scriptsize 
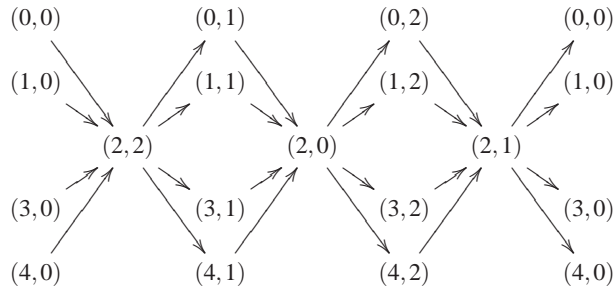
\[\xymatrix@C=8pt@R=8pt{
(0,0)\ar[rdd]&&(0,1)\ar[rdd]&&(0,2)\ar[rdd]&&(0,0) \\
(1,0)\ar[rd]&&(1,1)\ar[rd]&&(1,2)\ar[rd]&&(1,0) \\
&(2,2)\ar[ruu]\ar[ru]\ar[rdd]\ar[rd]&&(2,0)\ar[ruu]\ar[ru]\ar[rdd]\ar[rd]&&(2,1)\ar[ruu]\ar[ru]\ar[rdd]\ar[rd]& \\
(3,0)\ar[ru]&&(3,1)\ar[ru]&&(3,2)\ar[ru]&&(3,0) \\
(4,0)\ar[ruu]&&(4,1)\ar[ruu]&&(4,2)\ar[ruu]&&(4,0)
}\] }

where the left and right hand sides are identified and the vertex $(0,0)$ represents $R$ $($for more details, see \cite{AR1}$)$. 
In this case, we can see that $E=M_{2,2}, \,\tau^{-1}R=M_{0,1}$. $($Check the notation used in the above algorithm.$)$ 
By applying the counting argument to this quiver, we have the following. 

{\tiny 
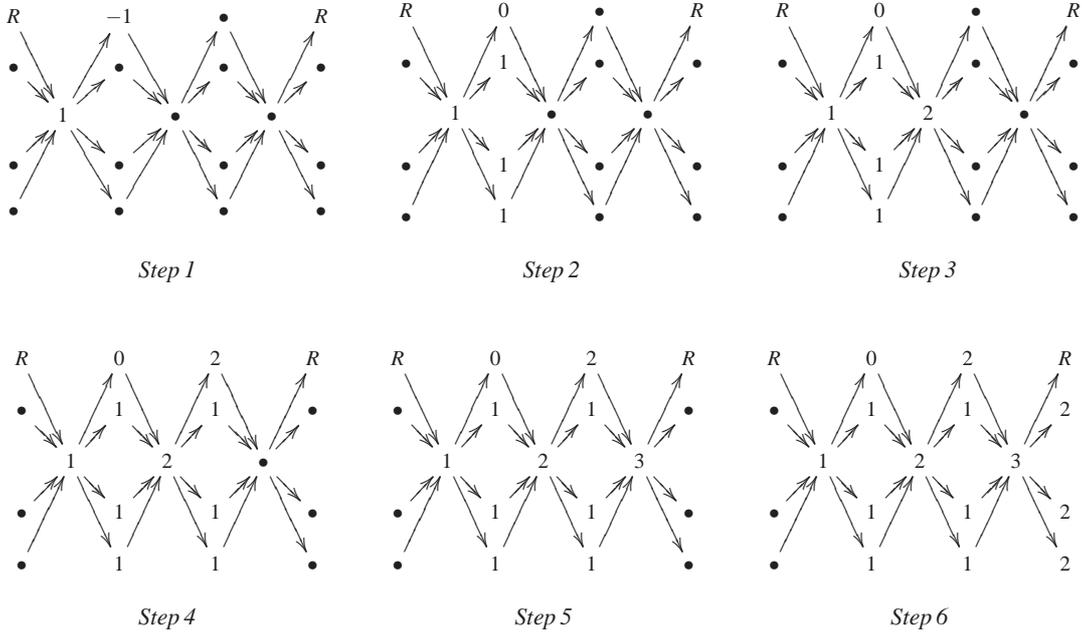
\[
 \begin{array}{ccc}
 \begin{array}{c}
 \xymatrix@C=7pt@R=7pt{
 R\ar[rdd]&&-1\ar[rdd]&&\bullet\ar[rdd]&&R \\
 \bullet\ar[rd]&&\bullet\ar[rd]&&\bullet\ar[rd]&&\bullet \\
 &1\ar[ruu]\ar[ru]\ar[rdd]\ar[rd]&&\bullet\ar[ruu]\ar[ru]\ar[rdd]\ar[rd]&&\bullet\ar[ruu]\ar[ru]\ar[rdd]\ar[rd]& \\
 \bullet\ar[ru]&&\bullet\ar[ru]&&\bullet\ar[ru]&&\bullet \\
 \bullet\ar[ruu]&&\bullet\ar[ruu]&&\bullet\ar[ruu]&&\bullet}
 \end{array} &
 \begin{array}{c}
 \xymatrix@C=7pt@R=7pt{
 R\ar[rdd]&&0\ar[rdd]&&\bullet\ar[rdd]&&R \\
 \bullet\ar[rd]&&1\ar[rd]&&\bullet\ar[rd]&&\bullet \\
 &1\ar[ruu]\ar[ru]\ar[rdd]\ar[rd]&&\bullet\ar[ruu]\ar[ru]\ar[rdd]\ar[rd]&&\bullet\ar[ruu]\ar[ru]\ar[rdd]\ar[rd]& \\
 \bullet\ar[ru]&&1\ar[ru]&&\bullet\ar[ru]&&\bullet \\
 \bullet\ar[ruu]&&1\ar[ruu]&&\bullet\ar[ruu]&&\bullet} 
 \end{array} &
 \begin{array}{c}
 \xymatrix@C=7pt@R=7pt{
 R\ar[rdd]&&0\ar[rdd]&&\bullet\ar[rdd]&&R \\
 \bullet\ar[rd]&&1\ar[rd]&&\bullet\ar[rd]&&\bullet \\
 &1\ar[ruu]\ar[ru]\ar[rdd]\ar[rd]&&2\ar[ruu]\ar[ru]\ar[rdd]\ar[rd]&&\bullet\ar[ruu]\ar[ru]\ar[rdd]\ar[rd]& \\
 \bullet\ar[ru]&&1\ar[ru]&&\bullet\ar[ru]&&\bullet \\
 \bullet\ar[ruu]&&1\ar[ruu]&&\bullet\ar[ruu]&&\bullet}
 \end{array} \\
 &&\\
 \mbox{\scriptsize Step\,1}&\mbox{\scriptsize Step\,2}&\mbox{\scriptsize Step\,3}
\end{array}
\]  
\[
 \begin{array}{ccc}
 \begin{array}{c}
 \xymatrix@C=7pt@R=7pt{
 R\ar[rdd]&&0\ar[rdd]&&2\ar[rdd]&&R \\
 \bullet\ar[rd]&&1\ar[rd]&&1\ar[rd]&&\bullet \\
 &1\ar[ruu]\ar[ru]\ar[rdd]\ar[rd]&&2\ar[ruu]\ar[ru]\ar[rdd]\ar[rd]&&\bullet\ar[ruu]\ar[ru]\ar[rdd]\ar[rd]& \\
 \bullet\ar[ru]&&1\ar[ru]&&1\ar[ru]&&\bullet \\
 \bullet\ar[ruu]&&1\ar[ruu]&&1\ar[ruu]&&\bullet}
 \end{array} &
 \begin{array}{c}
 \xymatrix@C=7pt@R=7pt{
 R\ar[rdd]&&0\ar[rdd]&&2\ar[rdd]&&R \\
 \bullet\ar[rd]&&1\ar[rd]&&1\ar[rd]&&\bullet \\
 &1\ar[ruu]\ar[ru]\ar[rdd]\ar[rd]&&2\ar[ruu]\ar[ru]\ar[rdd]\ar[rd]&&3\ar[ruu]\ar[ru]\ar[rdd]\ar[rd]& \\
 \bullet\ar[ru]&&1\ar[ru]&&1\ar[ru]&&\bullet \\
 \bullet\ar[ruu]&&1\ar[ruu]&&1\ar[ruu]&&\bullet} 
 \end{array} &
 \begin{array}{c}
 \xymatrix@C=7pt@R=7pt{
 R\ar[rdd]&&0\ar[rdd]&&2\ar[rdd]&&R \\
 \bullet\ar[rd]&&1\ar[rd]&&1\ar[rd]&&2 \\
 &1\ar[ruu]\ar[ru]\ar[rdd]\ar[rd]&&2\ar[ruu]\ar[ru]\ar[rdd]\ar[rd]&&3\ar[ruu]\ar[ru]\ar[rdd]\ar[rd]& \\
 \bullet\ar[ru]&&1\ar[ru]&&1\ar[ru]&&2 \\
 \bullet\ar[ruu]&&1\ar[ruu]&&1\ar[ruu]&&2}  
 \end{array} \\
 &&\\
 \mbox{\scriptsize Step\,4}&\mbox{\scriptsize Step\,5}&\mbox{\scriptsize Step\,6}
\end{array}
\] }

Continuing with this process, finally we get to the following picture.  

{\scriptsize 
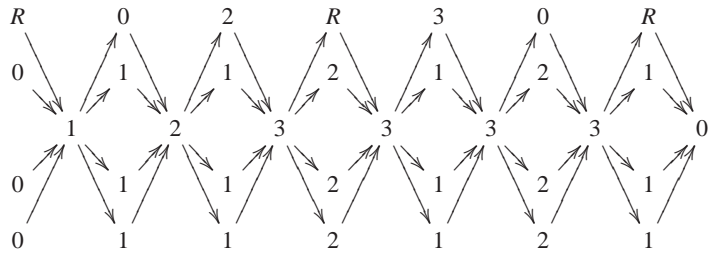
\[\xymatrix@C=8pt@R=8pt{
 R\ar[rdd]&&0\ar[rdd]&&2\ar[rdd]&&R\ar[rdd]&&3\ar[rdd]&&0\ar[rdd]&&R\ar[rdd]& \\
 0\ar[rd]&&1\ar[rd]&&1\ar[rd]&&2\ar[rd]&&1\ar[rd]&&2\ar[rd]&&1\ar[rd]& \\
 &1\ar[ruu]\ar[ru]\ar[rdd]\ar[rd]&&2\ar[ruu]\ar[ru]\ar[rdd]\ar[rd]&&3\ar[ruu]\ar[ru]\ar[rdd]\ar[rd]&&3\ar[ruu]\ar[ru]\ar[rdd]\ar[rd]&&3\ar[ruu]\ar[ru]\ar[rdd]\ar[rd]&&3\ar[ruu]\ar[ru]\ar[rdd]\ar[rd]&&0\\
 0\ar[ru]&&1\ar[ru]&&1\ar[ru]&&2\ar[ru]&&1\ar[ru]&&2\ar[ru]&&1\ar[ru]& \\
 0\ar[ruu]&&1\ar[ruu]&&1\ar[ruu]&&2\ar[ruu]&&1\ar[ruu]&&2\ar[ruu]&&1\ar[ruu]&
}\] } 

By extracting non-zero paths from the above quiver, we have the Figure\,\ref{D_52_quiver} where the exponent of each vertex implies the multiplicity.  

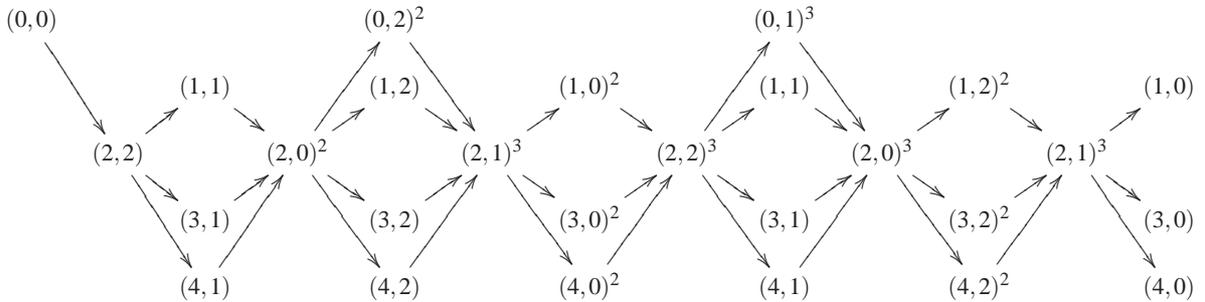
\begin{figure}[!h]
{\scriptsize 
\[\xymatrix@C=6pt@R=8pt{
(0,0)\ar[rdd]&&&&(0,2)^2\ar[rdd]&&&&(0,1)^3\ar[rdd]&&&& \\
&&(1,1)\ar[rd]&&(1,2)\ar[rd]&&(1,0)^2\ar[rd]&&(1,1)\ar[rd]&&(1,2)^2\ar[rd]&&(1,0) \\
&(2,2)\ar[ru]\ar[rdd]\ar[rd]&&(2,0)^2\ar[ruu]\ar[ru]\ar[rdd]\ar[rd]&&(2,1)^3\ar[ru]\ar[rdd]\ar[rd]&&(2,2)^3\ar[ruu]\ar[ru]\ar[rdd]\ar[rd]&&(2,0)^3\ar[ru]\ar[rdd]\ar[rd]&&(2,1)^3\ar[ru]\ar[rdd]\ar[rd]& \\
&&(3,1)\ar[ru]&&(3,2)\ar[ru]&&(3,0)^2\ar[ru]&&(3,1)\ar[ru]&&(3,2)^2\ar[ru]&&(3,0) \\
&&(4,1)\ar[ruu]&&(4,2)\ar[ruu]&&(4,0)^2\ar[ruu]&&(4,1)\ar[ruu]&&(4,2)^2\ar[ruu]&&(4,0)
}\] }
\caption{Composition $R\rightarrow M_{2,2}$  and non-zero elements of $\underline{\Hom}(M_{2,2},-)$ }
\label{D_52_quiver}
\end{figure}
\end{ex}

Thus, we can identify minimal generators of $M_{i,j}$ with non-zero paths from $R$ to $M_{i,j}$ on Figure\ref{D_52_quiver}. 
For example, $M_{1,1}$ has two minimal generators associated to the following two paths.

\[
\begin{array}{ccc}
{\scriptsize 
\xymatrix@C=4pt@R=5pt{
R\ar[rdd]&& \\
&&(1,1) \\
&\bullet\ar[ru]& \\
&&\bullet \\
&&\bullet
} }& 
\xymatrix@C=4pt@R=5pt{
&& \\
&& \\
&\text{and}& \\
&& \\
&&
} &
{\scriptsize 
\xymatrix@C=4pt@R=5pt{
R\ar[rdd]&&&&\bullet&&&&\bullet \\
&&\bullet&&\bullet\ar[rd]&&\bullet&&(1,1) \\
&\bullet\ar[rd]&&\bullet\ar[ru]&&\bullet\ar[rd]&&\bullet\ar[ru]& \\
&&\bullet\ar[ru]&&\bullet&&\bullet\ar[ru]&&\bullet \\
&&\bullet&&\bullet&&\bullet&&\bullet
} }
\end{array}
\]

\bigskip

In other words, a minimal generator of $R$ $($i.e. a unit of $R)$ generates those of $M_1$ by chasing the above paths.  
Of course, there are several paths from $R$ to $M_{1,1}$ not only the above ones. 
Since the AR quiver has relations originated from AR sequences, 
they generate the same minimal generator up to modulo radical. 
Furthermore, composing the irreducible morphism $M_{1,1}\rightarrow M_{2,0}$ and the above paths, 
we have a part of minimal generators of $M_{2,0}$. 

\begin{rem}
From the constructing method of the AR sequence, each arrow in the AR quiver is a degree $1$ map induced by the standard grading of $k[[x,y]]$. 
Therefore, the degree of each minimal generator of $M$ coincides with the associated path length from $R$ to $M$.
\end{rem}

\section{Key lemma for determining the dual $F$-signature}
\label{Sec_keylemma}

In order to investigate a surjection from ${}^eM_t$ to a finite direct sum of some copies of $M_t$, we will prepare a technical lemma. 

As we noted in the beginning of subsection\,\ref{count_AR}, we may consider an MCM $R$-module $M_t$ as a vector space. 
More precisely, let $M_i, M_j$ be indecomposable MCM modules and suppose a morphism $\varphi_i:M_i\rightarrow M_j$ is 
a non-zero path appearing in the AR quiver after applying the counting argument. Then, $\Image \varphi_i$ constructs part of the minimal generators of $M_j$.
Therefore in the commutative diagram 
 \[\xymatrix@C=10pt@R=10pt{
   M_i\ar[r]^{\varphi_i}\ar[dd]&M_j\ar[dd]^{\pi}\\
   &\\
   V_i\coloneqq \Image(\varphi_i\pi)\ar@{^{(}->}[r]&V_j\coloneqq M_j/\fkm M_j,  
}\]
we may consider $V_i$ as a vector subspace of $V_j$ and take a injective morphism 
\[
 X_i\cdot1_{V_i}:V_i\hookrightarrow V_j \quad(X_i\in k).
\]

Now we prove the key lemma related to these vector spaces in more general settings (Lemma\,\ref{key_lemma}).
Let $V$ be a $d$-dimensional $k$-vector space and fix a basis $\{v_1,\cdots,v_d\}$. 
Suppose $W_1,\cdots,W_r$ are subspaces of $V$ (admit repetition) where $\dim_kW_i=d_i\le d$ and 
the basis of $W_i$ is the part of $\{v_1,\cdots,v_d\}$. 
Namely, we choose $d_i$ elements from $\{v_1,\cdots,v_d\}$ as the basis of $W_i$.
Define the $d\times r$ table $[a_{ij}]$ associated with $W_i$'s as follows,
\begin{equation*}
a_{ij}=\begin{cases} 1&\text{(if $v_i$ is a basis of $W_j$)} \\
                     0&\text{(if $v_i$ is not a basis of $W_j$)} 
           \end{cases}
\end{equation*}
where $i=1,\cdots,d$ and $j=1,\cdots,r$.

\begin{lem}
\label{key_lemma}
Set $n\coloneqq \mathrm{min}\{\sum^r_{j=1}a_{ij}\;|\;i=1,\cdots,d\;\}\le r$, then 
there exists a surjection 
\[
 W_1\oplus\cdots\oplus W_r \twoheadlongrightarrow V^{\oplus n}.
\]
\end{lem}

\begin{proof}
Firstly, we define a $(dn)\times(dr)$ matrix 
{\small
\begin{equation*}
C=
\begin{pmatrix}
A^{(1)}_1&A^{(1)}_2&\cdots&A^{(1)}_r \\
A^{(2)}_1&A^{(2)}_2&\cdots&A^{(2)}_r \\
\hdotsfor{4}\\
A^{(n)}_1&A^{(n)}_2&\cdots&A^{(n)}_r 
\end{pmatrix}
\end{equation*}}
where $A^{(\alpha)}_j\;(1\le\alpha\le n,\;1\le j\le n)$ is a $d\times d$ diagonal matrix as follows.
{\small
\begin{equation*}
A^{(\alpha)}_j=X^{(\alpha)}_j
\begin{pmatrix}
a_{1j}&& \\
&a_{2j}& \\
&&\ddots \\
&&&a_{nj}
\end{pmatrix},
\quad \text{where} \;X^{(\alpha)}_j\in k.
\end{equation*} }
Especially, we can take $X^{(\alpha)}_j\;(1\le\alpha\le n,\;1\le j\le n)$ as algebraically independent variables.
Note that for the vector space $V^{\oplus n}=V^{(1)}\oplus\cdots\oplus V^{(n)}$ where $V^{(\alpha)}\cong V$ and linear morphisms 
$W_j\overset{\cdot X^{(\alpha)}_j}{\hookrightarrow} V^{(\alpha)}$, the matrix $C=(c_{st})$ is a representation matrix of 
$\varphi:W_1\oplus\cdots\oplus W_r \longrightarrow V^{\oplus n}$.
Thus, $\varphi$ is surjective if and only if there exists a nonzero $dn$-minor of $C$. 
From now on, we construct such a $dn$-minor. 

For this purpose, we choose $dn$ columns which are distinct from each other from $C$ and consider a sequence $(t_1,t_2,\cdots,t_{dn})$ 
where $1\le t_1,\cdots,t_{dn}\le dn$ are column numbers. 
From a sequence $(t_1,t_2,\cdots,t_{dn})$ of $C$, we obtain the monomial in a natural fashion, 
\[
(t_1,t_2,\cdots,t_{dn})\mapsto\prod^{dn}_{s=1}c_{s,t_s}\in \mathrm{Mon}(X^{(\alpha)}_j\;|\;1\le j\le r,\;1\le \alpha\le n)
\]
where $\mathrm{Mon}(X^{(\alpha)}_j)$ is the monomial set of $k[X^{(\alpha)}_j]$. 
We say a sequence $(t_1,t_2,\cdots,t_{dn})$ is chain if the corresponding monomial is not zero. 
We impose the lexicographic order
\begin{equation}
\label{monomial_order}
X^{(1)}_1>\cdots>X^{(1)}_r>X^{(2)}_1>\cdots>X^{(2)}_r>\cdots>X^{(n)}_1>\cdots>X^{(n)}_r
\end{equation}
on $\mathrm{Mon}(X^{(\alpha)}_j)$. 

From now on, we consider the chain of $C$ constructed from the following algorithm. 

\bigskip

\begin{itemize}
 \item [\textbf{(Step1)}] 
For the $d\times r$ table $[a_{ij}]$, if there is a number $i$ such that $a_{ij}=0$ for all $j=1,\cdots,r$, 
then we stop this operation $($Namely, if $n=0$ then we stop here$)$. Otherwise, we set 
\[
j^{(1)}_1\coloneqq \operatorname{min}\{j\;|\;a_{1j}=1\} \quad\text{and}\quad 
t_1\coloneqq 1+(j^{(1)}_1-1)d .
\]
After that, we replace the number $a_{1,j^{(1)}_1}=1$ by $0$. 

Similarly, we set 
\[
j^{(1)}_2\coloneqq \operatorname{min}\{j\;|\;a_{2j}=1\} \quad\text{and}\quad 
t_2\coloneqq 2+(j^{(1)}_2-1)d ,
\]
and replace $a_{2,j^{(1)}_2}=1$ by $0$. 
\begin{center}
$\vdots$
\end{center}
We set 
\[
j^{(1)}_d\coloneqq \operatorname{min}\{j\;|\;a_{dj}=1\} \quad\text{and}\quad 
t_d\coloneqq d+(j^{(1)}_d-1)d ,
\]
and replace $a_{d,j^{(1)}_d}=1$ by $0$, then we stop (Step1) here.
\begin{center}
$\bullet$ \\
$\bullet$ \\
$\bullet$
\end{center}
 \item [\textbf{(Step\;$\alpha$)}] 
For the $d\times r$ table $[a_{ij}]$, if there is a number $i$ such that $a_{ij}=0$ for all $j=1,\cdots,r$, 
then we stop this operation $($Namely, if $\alpha>n$, then we stop here$)$. Otherwise, we set 
\[
j^{(\alpha)}_1\coloneqq \operatorname{min}\{j\;|\;a_{1j}=1\} \quad\text{and}\quad 
t_{d(\alpha-1)+1}\coloneqq 1+(j^{(\alpha)}_1-1)d .
\]
replace $a_{1,j^{(\alpha)}_1}=1$ by $0$. 
\begin{center}
$\vdots$
\end{center}
We set 
\[
j^{(\alpha)}_d\coloneqq \operatorname{min}\{j\;|\;a_{dj}=1\} \quad\text{and}\quad 
t_{d(\alpha-1)+d}\coloneqq d+(j^{(\alpha)}_d-1)d ,
\]
and replace $a_{d,j^{(\alpha)}_d}=1$ by $0$, then we stop (Step\;$\alpha$) here.
\begin{center}
$\bullet$ \\
$\bullet$ \\
$\bullet$ \\
(repeat this process up to Step $n$) 
\end{center}
\end{itemize}

\medskip

By the definition of the number $n$, we can repeat this process up to $($Step\;$n)$. 
After that, we have $a_{i1}=\cdots=a_{ir}=0$ for some $i$.  
Therefore, we stop this algorithm. 

\bigskip

From the above operation, we obtain the sequence $(t_1,t_2,\cdots,t_{dn})$ and 
this sequence is clearly a chain by the construction method. 
Finally, we prove the following Claim\,\ref{key_lemma+} and complete the proof of Lemma\,\ref{key_lemma}.
\end{proof}

\begin{cla}
\label{key_lemma+}
The $dn$-minor $[t_1,t_2,\cdots,t_{dn}]$ of $C$ is non-zero.
\end{cla}

\begin{proof}
Define $dn\times dn$-matrix $D=(d_{st})$ by choosing the columns $t_1,t_2,\cdots,t_{dn}$ from $C$.
By the definition of determinant 
\[
[t_1,t_2,\cdots,t_{dn}]=\operatorname{det}D=\sum_{\sigma\in\fkS_{dn}}(\operatorname{sgn}\sigma)d_{1,\sigma(1)}\cdots d_{dn,\sigma(dn)} ,
\]
where $\fkS_{dn}$ is a symmetric group of degree $dn$. 
From the selecting method of $t_1,\cdots,t_{dn}$, the monomial $0\neq\prod^{dn}_{s=1}c_{s,t_s}$ appears in the monomial set 
$\{d_{1,\sigma(1)}\cdots d_{dn,\sigma(dn)}\}_{\sigma\in\fkS_{dn}}$ and it is the unique maximal element with respect to the lexicographic order (\ref{monomial_order}). 
Furthermore, the algebraically independence of $X_j^{(\alpha)}s$ implies $\operatorname{det}D\neq 0$.
\end{proof}

\begin{ex}
Let $V$ be a $3$-dimensional vector space over $k$ and fix a basis $\{v_1,v_2,v_3\}$. 
Consider subspaces of $V$ ;
\[
W_1=<v_1,v_2>,\quad W_2=<v_2,v_3>,\quad W_3=<v_1>,\quad W_4=<v_1,v_3> .
\]
\begin{center}
$[a_{ij}]=$
 {\scriptsize\begin{tabular}{r|cccc} 
 &$W_1$&$W_2$&$W_3$&$W_4$ \\ \hline
 $v_1$&$1$&$0$&$1$&$1$ \\
 $v_2$&$1$&$1$&$0$&$0$ \\
 $v_3$&$0$&$1$&$0$&$1$ \\
 \end{tabular}}
\end{center}

By Lemma\,\ref{key_lemma}, we have a surjection $W_1\oplus W_2\oplus W_3\oplus W_4 \twoheadrightarrow V^{\oplus 2}$. 

Note that $(t_1,\cdots,t_6)=(1,2,6,7,5,12)$ and $\prod^6_{s=1}c_{s,t_s}$ is just the product of underlined entries of $C$.
 {\footnotesize\[
C=\left(
 \begin{array}{@{\,}ccc|ccc|ccc|ccc@{\,}}
  \underline{X^{(1)}_1}&&&0&&&X^{(1)}_3&&&X^{(1)}_4&& \\
 &\underline{X^{(1)}_1}&&&X^{(1)}_2&&&0&&&0& \\ 
 &&0&&&\underline{X^{(1)}_2}&&&0&&&X^{(1)}_4 \\ [5pt] \hline
 \rule[5pt]{0pt}{15pt}
 X^{(2)}_1&&&0&&&\underline{X^{(2)}_3}&&&X^{(2)}_4&& \\
 &X^{(2)}_1&&&\underline{X^{(2)}_2}&&&0&&&0& \\
 &&0&&&X^{(2)}_2&&&0&&&\underline{X^{(2)}_4}
 \end{array}
\right)\]}
\end{ex}

\section{Dual $F$-signature over rational double points}
\label{DFsig_RDP}

Firstly, we recall some well-known facts about two-dimensional rational double points.
We suppose that $G$ is a finite subgroup of $\SL(2,k)$ and the order of $G$ is coprime to $p=\mathrm{char}\,k$. 
We denote the invariant subring of $S\coloneqq k[[x,y]]$ under the action of $G$ by $R\coloneqq S^G$ and the maximal ideal of $R$ by $\fkm$. 
In this situation, the invariant subring $R$ is Gorenstein by \cite{Wat}. 
We call $R$ $($or equivalently $\Spec R)$ rational double points (or Du Val singularities, Kleinian singularities, ADE singularities in the literature). 
Moreover, we can see that $G$ contains no pseudo-reflections except the identity in this situation. Thus, the results in Section\,\ref{Gene_Fsig} hold for $R$ and 
it is well known that a finite subgroup of $\SL(2,k)$ is 
conjugate to one of the following finite groups (e.g. \cite[Chapter 10]{Yo}); 
\begin{equation}
\label{sub_SL2}
 \begin{array}{cl}
  (A_n):&\text{the cyclic group of order $n+1$} \;\;(n\ge 1) \\
&\quad\calC_{n+1}\coloneqq \langle\;
    \begin{pmatrix} \zeta_{n+1}&0 \\
                    0&\zeta_{n+1}^{-1} 
    \end{pmatrix}           \;\rangle \\
  (D_n):&\text{the binary dihedral group of order $4(n-2)$} \;\;(n\ge 4) \\
&\quad\calD_{n-2}\coloneqq \langle\;
   \calC_{2(n-2)},\; 
    \begin{pmatrix} 0&\zeta_4 \\
                    \zeta_4&0
    \end{pmatrix}         \;\rangle \\ 
  (E_6):&\text{the binary tetrahedral group of order $24$}  \\
&\quad\calT \coloneqq \langle\;
    \displaystyle\frac{1}{\sqrt{2}}\begin{pmatrix} \zeta_8&\zeta_8^3 \\
                    \zeta_8&\zeta_8^7
    \end{pmatrix},\; \calD_2        \;\rangle \\
  (E_7):&\text{the binary octahedral group of order $48$}  \\
&\quad\calO \coloneqq \langle\;
    \begin{pmatrix} \zeta_8^3&0 \\
                    0&\zeta_8^5 
    \end{pmatrix},\; \calT    \;\rangle \\
  (E_8):&\text{the binary icosahedral group of order $120$} \\
&\quad\calI \coloneqq \langle\;
    \displaystyle\frac{1}{\sqrt{5}}\begin{pmatrix} \zeta_5^4-\zeta_5&\zeta_5^2-\zeta_5^3 \\
                    \zeta_5^2-\zeta_5^3&\zeta_5-\zeta_5^4
    \end{pmatrix},\;
    \displaystyle\frac{1}{\sqrt{5}}\begin{pmatrix} \zeta_5^2-\zeta_5^4&\zeta_5^4-1 \\
                    1-\zeta_5&\zeta_5^3-\zeta_5
    \end{pmatrix}
\;\rangle 
 \end{array}
\end{equation}
where $\zeta_n$ is a primitive $n$-th root of unity. 

\medskip

Moreover, the AR quiver of $R$ coincides with the extended Dynkin diagram corresponding to the above types 
after replacing each edges $``-"$ by arrows $``\leftrightarrows"$.

Therefore the Auslander-Reiten quiver of $R$ is the left hand side of the following, 
{\scriptsize 
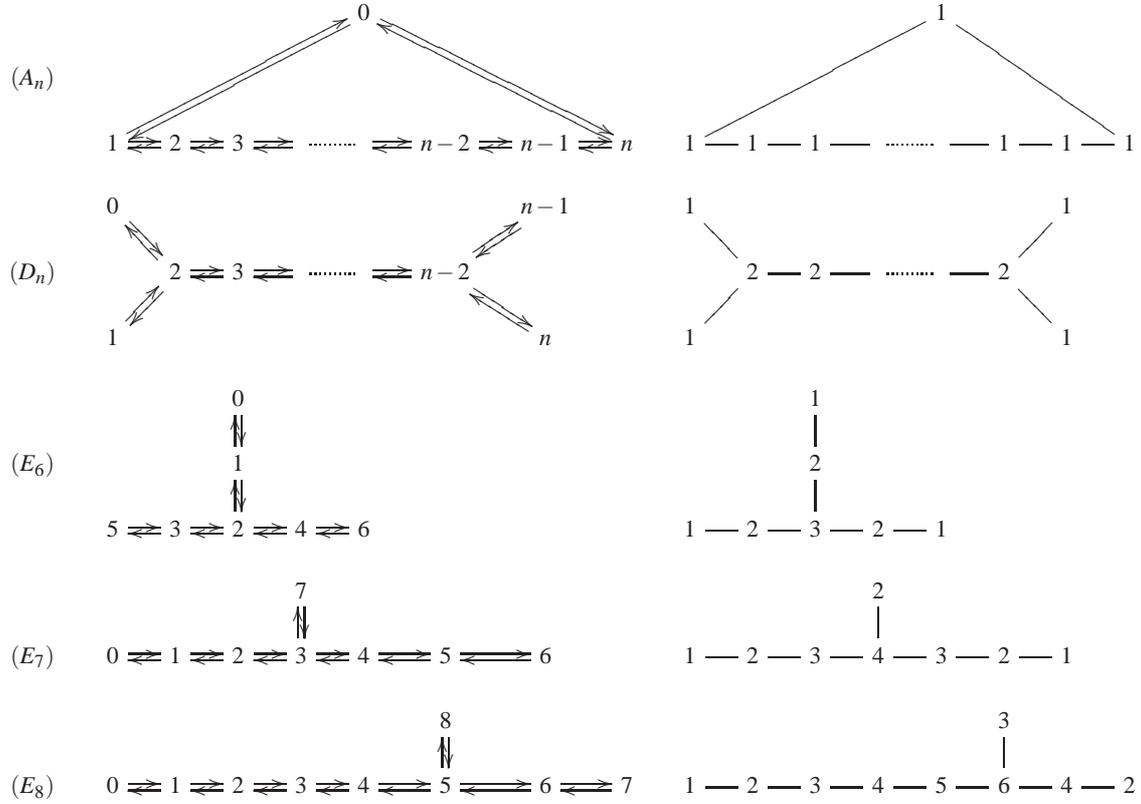
\begin{equation*}
\label{ARquiver_AED}
\xymatrix@C=12pt@R=10pt{
 &&&&&0\ar@<0.3ex>[rrrdd]\ar@<0.3ex>[lllldd]&&& & &&&&1\ar@{-}[rrrdd]\ar@{-}[lllldd]&&& \\
 (A_n) &&&&&&&& & &&&&&&& \\
 &1\ar@<0.3ex>[r]\ar@<0.3ex>[rrrruu]&2\ar@<0.3ex>[l]\ar@<0.3ex>[r]&3\ar@<0.3ex>[r]\ar@<0.3ex>[l]&\ar@<0.3ex>[l]\ar@{..}[r]&\ar@<0.3ex>[r]&n-2\ar@<0.3ex>[r]\ar@<0.3ex>[l]&n-1\ar@<0.3ex>[r]\ar@<0.3ex>[l]&n\ar@<0.3ex>[l]\ar@<0.3ex>[llluu] &
 1\ar@{-}[r]&1\ar@{-}[r]&1\ar@{-}[r]&\ar@{..}[r]&\ar@{-}[r]&1\ar@{-}[r]&1\ar@{-}[r]&1 \\
 &0\ar@<0.3ex>[rd]&&&&&&n-1\ar@<0.3ex>[ld] && 1\ar@{-}[rd]&&&&&&1\ar@{-}[ld]\\
 (D_n)&&2\ar@<0.3ex>[r]\ar@<0.3ex>[lu]\ar@<0.3ex>[ld]&3\ar@<0.3ex>[r]\ar@<0.3ex>[l]&\ar@{..}[r]\ar@<0.3ex>[l]&\ar@<0.3ex>[r]&n-2\ar@<0.3ex>[l]\ar@<0.3ex>[ru]\ar@<0.3ex>[rd] &&&
 &2\ar@{-}[r]&2\ar@{-}[r]&\ar@{..}[r]&\ar@{-}[r]&2 \\
 &1\ar@<0.3ex>[ru]&&&&&&n\ar@<0.3ex>[lu] && 1\ar@{-}[ru]&&&&&&1\ar@{-}[lu] \\
 &&&0\ar@<0.3ex>[d] &&&&& & &&1\ar@{-}[d]\\
 (E_6)&&&1\ar@<0.3ex>[u]\ar@<0.3ex>[d]&&&&& & &&2\ar@{-}[u]\ar@{-}[d]\\
 &5\ar@<0.3ex>[r]&3\ar@<0.3ex>[r]\ar@<0.3ex>[l]&2\ar@<0.3ex>[r]\ar@<0.3ex>[u]\ar@<0.3ex>[l]&4\ar@<0.3ex>[r]\ar@<0.3ex>[l]&6\ar@<0.3ex>[l] &&& &
 1\ar@{-}[r]&2\ar@{-}[r]&3\ar@{-}[r]&2\ar@{-}[r]&1\\
 &&&&7\ar@<0.3ex>[d] &&&& & &&&2\ar@{-}[d]\\
 (E_7)&0\ar@<0.3ex>[r]&1\ar@<0.3ex>[r]\ar@<0.3ex>[l]&2\ar@<0.3ex>[r]\ar@<0.3ex>[l]&3\ar@<0.3ex>[r]\ar@<0.3ex>[u]\ar@<0.3ex>[l]&4\ar@<0.3ex>[r]\ar@<0.3ex>[l]&5\ar@<0.3ex>[r]\ar@<0.3ex>[l]&6\ar@<0.3ex>[l]& &
 1\ar@{-}[r]&2\ar@{-}[r]&3\ar@{-}[r]&4\ar@{-}[r]&3\ar@{-}[r]&2\ar@{-}[r]&1\\
 &&&&&&8\ar@<0.3ex>[d]&& & &&&&&3\ar@{-}[d]\\
 (E_8)&0\ar@<0.3ex>[r]&1\ar@<0.3ex>[r]\ar@<0.3ex>[l]&2\ar@<0.3ex>[r]\ar@<0.3ex>[l]&3\ar@<0.3ex>[r]\ar@<0.3ex>[l]&4\ar@<0.3ex>[r]\ar@<0.3ex>[l]&5\ar@<0.3ex>[r]\ar@<0.3ex>[u]\ar@<0.3ex>[l]&6\ar@<0.3ex>[r]\ar@<0.3ex>[l]&7\ar@<0.3ex>[l]&
 1\ar@{-}[r]&2\ar@{-}[r]&3\ar@{-}[r]&4\ar@{-}[r]&5\ar@{-}[r]&6\ar@{-}[r]&4\ar@{-}[r]&2
} 
\end{equation*} }

where a vertex $t$ corresponds the MCM $R$-module $M_t$ and the right hand side of the figure means $\rank_RM_t$.

\begin{rem}
\label{rem_multi}
From Corollary \ref{gene-Fsig-cor}, we may consider as
\[
{}^eM_t\approx (R^{\oplus d_{0,t}}\oplus M_1^{\oplus d_{1,t}}\oplus\cdots\oplus M_n^{\oplus d_{n,t}})^{\oplus\frac{p^{2e}}{|G|}},
\]
where $d_{i,t}=(\rank_RM_t)\cdot(\rank_RM_i)$. 
When we try to determine the dual $F$-signature, the part of $o(p^{2e})$ is harmless. 
Therefore, we identify ${}^eM_t$ with $R^{\oplus d_{0,t}/|G|}\oplus M_1^{\oplus d_{1,t}/|G|}\oplus\cdots\oplus M_n^{\oplus d_{n,t}/|G|}$ and
sometimes omit $|G|^{-1}$ for simplicity. 

For reasons of showing the ratio of $s(M_t)$ to $|G|$ clearly, we don't reduce a fraction. 
\end{rem}

In order to determine the value of the dual $F$-signature, we need understand the paths which generate minimal generators 
by applying the counting argument of AR quiver.
As the counting argument written below shows, the number of minimal generators of $M_t$ is equal to $m_t\coloneqq 2\rank_RM_t$
(see also \cite[Theorem\,1.2]{Wun2}).
We denote minimal generators of $M_t$ by $g_{t,1},g_{t,2},\cdots,g_{t,m_t}$ 
and assume $\deg g_{t,1}\le\deg g_{t,2}\le\cdots\le\deg g_{t,m_t}$.

\subsection{Type $A_n$} 
\label{type_An}

The author determined the dual $F$-signature of $A_n$ type in \cite[Example\,3.11]{Nak} as follows. 

\begin{equation*}
s(M_t)=\begin{cases} \displaystyle\frac{t+1}{n+1}&\text{(if \;$t<\frac{n+1}{2}$)} \\
                         &\\
                         \displaystyle\frac{2t+1}{2(n+1)}&\text{(if \;$t=\frac{n+1}{2}$)} \\
                         &\\
                         \displaystyle\frac{n-t+2}{n+1}&\text{(if \;$t>\frac{n+1}{2}$)}. 
           \end{cases}
\end{equation*}

\subsection{Type $D_n$}
\label{type_Dn} 

Firstly, we show a method for determining the dual $F$-signatures in the case of type $D_5$ as an example. 
This method also applies to other cases. 

\medskip

\begin{ex}
The binary dihedral group $
G\coloneqq\calD_3= {\small\langle\;
    \begin{pmatrix} \zeta_6&0 \\
                    0&\zeta_6^{-1} 
    \end{pmatrix},\; 
    \begin{pmatrix} 0&\zeta_4 \\
                    \zeta_4&0
    \end{pmatrix}         \;\rangle} $ 
is the type $D_5$ in the list\;$($\ref{sub_SL2}$)$ and $|G|=12$. 
For the invariant subring under the action of $G$, the AR quiver takes the form as follows. 
{\small 
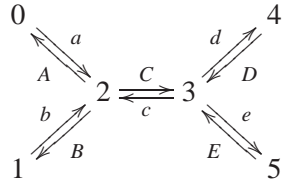
\[\xymatrix@C=20pt@R=15pt{
 0\ar@<0.3ex>[rd]^a&&&4\ar@<0.3ex>[ld]^D\\
 &2\ar@<0.3ex>[r]^C\ar@<0.3ex>[lu]^A\ar@<0.3ex>[ld]^B&3\ar@<0.3ex>[l]^c\ar@<0.3ex>[ru]^d\ar@<0.3ex>[rd]^e \\
 1\ar@<0.3ex>[ru]^b&&&5\ar@<0.3ex>[lu]^E \\
}\]}
It has the relations; 
\begin{equation}
\label{relation_D5}
\begin{cases}
aA=0,& cC+dD+eE=0, \\
bB=0,& Dd=0, \\
Aa+Bb+Cc=0,& Ee=0.
\end{cases}
\end{equation}

Rewriting this quiver as a repetition of the original one shown in dotted areas. 
Namely, we associate the translation quiver $\ZZ D_5$. $($The meaning of $\ZZ D_5$, see \cite{Gab}.$)$ 

{\small 
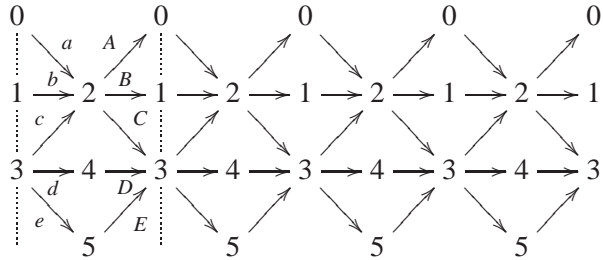
\[\xymatrix@C=15pt@R=15pt{
 0\ar[rd]^a\ar@{.}[d]&&0\ar[rd]\ar@{.}[d]&&0\ar[rd]&&0\ar[rd]&&0 \\
 1\ar[r]^b\ar@{.}[d]&2\ar[ru]^A\ar[r]^B\ar[rd]^C&1\ar[r]\ar@{.}[d]&2\ar[ru]\ar[r]\ar[rd]&1\ar[r]&2\ar[ru]\ar[r]\ar[rd]&1\ar[r]&2\ar[ru]\ar[r]\ar[rd]&1 \\
 3\ar[ru]^c\ar[r]_d\ar[rd]_e\ar@{.}[d]&4\ar[r]_D&3\ar[ru]\ar[r]\ar[rd]\ar@{.}[d]&4\ar[r]&3\ar[ru]\ar[r]\ar[rd]&4\ar[r]&3\ar[ru]\ar[r]\ar[rd]&4\ar[r]&3 \\
 &5\ar[ru]_E&&5\ar[ru]&&5\ar[ru]&&5\ar[ru]&
}\] }

After applying the counting argument $($cf. subsection\,\ref{count_AR}$)$, we have  

{\small 
\begin{equation}
\label{D5_path}
\xymatrix@C=15pt@R=15pt{
 0\ar[rd]&&&&&&& \\
 &2\ar[r]\ar[rd]&1\ar[r]&2\ar[rd]&&2\ar[r]\ar[rd]&1\ar[r]&2 \\
 &&3\ar[ru]\ar[r]\ar[rd]&4\ar[r]&3^2\ar[ru]\ar[r]\ar[rd]&4\ar[r]&3\ar[ru]& \\
 &&&5\ar[ru]&&5\ar[ru]&&
}
\end{equation} }

Thus, we identify paths on this quiver with minimal generators of each MCM module $M_t$. 

\bigskip

By using this, we will determine the dual $F$-signature of $M_1$ and $M_3$ as an example.

\medskip

\begin{itemize}
 \item[] \underline{\textbf{$\cdot$ the case of $M_1$ in $D_5$}\;}

 \medskip

 Since $\rank_RM_1=1$, the multiplicity $d_{t,1}$ of $M_t$ in ${}^eM_1$ is the following. 
Note that we consider them on the order of $p^{2e}$ and omit $p^{2e}/|G|$ times for simplicity $($see Remark\,\ref{rem_multi}$)$.
{\footnotesize 
 \begin{equation}
 \label{D5_M1_multi}
 \begin{tabular}{c|cccccc} 
   $M_t$&$R$&$M_1$&$M_2$&$M_3$&$M_4$&$M_5$ \\ \hline
   $d_{t,1}$&$1$&$1$&$2$&$2$&$1$&$1$ \\ 
 \end{tabular}
 \end{equation}}
Firstly, $R$ generates a minimal generator $g_{1,1}$ through the path $(R\xrightarrow{aB} M_1)$ on the quiver $(\ref{D5_path})$. 
Similarly, the paths $(M_1\xrightarrow{1_{M_1}}M_1)$ and $(M_2\xrightarrow{B}M_1)\times 2$ also generate $g_{1,1}$ $($Since $d_{2,1}=2$, we double the last one$)$ 
and we have no other such a path. 
Thus, the dual $F$-signature of $M_1$ can take 
$s(M_1)\le\frac{1}{12}+\frac{1}{12}+\frac{2}{12}=\frac{4}{12}$. 
So we obtain the upper bound of $s(M_1)$. Next, we will show that we can actually construct a surjection  
\[
R\oplus M_1\oplus M_2^{\oplus 2}\oplus M_3^{\oplus 2}\oplus M_4\oplus M_5\twoheadrightarrow M_1^{\oplus 4}.
\]
So if there exists such a surjection, 
then we can conclude $s(M_1)=\displaystyle\frac{4}{12}$.

From the quiver $($\ref{D5_path}$)$, we read off that $(M_2\xrightarrow{B}M_1)\times 2$ and $(M_3\xrightarrow{cB}M_1)\times 2$ generate $g_{1,2}$. 
Thus, we have the following table. As a consequence, we have the above surjection by Lemma\,\ref{key_lemma} and conclude $s(M_1)=\displaystyle\frac{4}{12}$. 
Note that a construction method of a surjection is not unique. It depends on a choice of paths.  
 {\footnotesize\[
 \begin{tabular}{c|cccccc|c} 
   $M_t$&$R$&$M_1$&$M_2$&$M_3$&$M_4$&$M_5$&Total \\ \hline
   Path&$a$&$1_{M_1}$&$B$&$cB$&$0$&$0$& \\ \hline
   $g_{1,1}$&$1$&$1$&$2$&$0$&$0$&$0$&$4$ \\
   $g_{1,2}$&$0$&$1$&$2$&$2$&$0$&$0$&$5$
 \end{tabular}
 \]}

\bigskip

 \item[] \underline{\textbf{$\cdot$ the case of $M_3$ in $D_5$}\;}

 \medskip

The strategy for determining $s(M_3)$ is the same as the case of $M_1$, 
but we need to pay attention to the central vertex $``3^2"$. As the multiplicity $2$ shows, paths from $R$ to this vertex 
could generate two kinds of minimal generator. 
Suppose that $\alpha\; ($resp. $\beta,\gamma)$ is a minimal generator of $M_3$ generated by a path which factor through 
$2\xrightarrow{C}3^2\; ($resp. $4\xrightarrow{D}3^2, 5\xrightarrow{E}3^2)$. 
By the relations $($\ref{relation_D5}$)$, they satisfy $\alpha+\beta+\gamma\in \fkm$ and we can take two of them as minimal generators 
associated to the vertex $3^2$. 
Thus, we fix $g_{3,2}\coloneqq\alpha,\; g_{3,3}\coloneqq\beta$. Since $\gamma$ is equivalent to $\alpha+\beta$ up to modulo radical, 
we use it freely as one of $\{\alpha, \beta\}$. 
Note that when we continue chasing a path after this vertex, we must not choose the following three paths, 
because the relations $($\ref{relation_D5}$)$ force them to be zero.
{\footnotesize 
\[\xymatrix@C=15pt@R=2pt{
 \ar[rd]^(0.38)C&&& & && & &3^2\ar[rd]^e& \\
 &3^2\ar[ru]^c&& & \ar[r]^(0.38)D&3^2\ar[r]^d& & \ar[ru]^(0.38)E&&
} \]}

 Since $\rank_RM_3=2$, the multiplicity $d_{t,3}$ of $M_t$ in ${}^eM_3$ is the following. 
{\footnotesize 
 \begin{equation}
 \label{D5_M3_multi}
 \begin{tabular}{c|cccccc} 
   $M_t$&$R$&$M_1$&$M_2$&$M_3$&$M_4$&$M_5$ \\ \hline
   $d_{t,3}$&$2$&$2$&$4$&$4$&$2$&$2$ \\ 
 \end{tabular}
 \end{equation}}
In order to estimate the upper bounds of $s(M_3)$, we consider paths which can be identified with $g_{3,1}$ or $g_{3,3}$. 
Then we classify each MCM modules as follows. 
\[
(\rmI)\;\{\;M_3\times 4\;\}\quad(\rmI\rmI)\;\{\;R\times 2, M_2\times 4\;\}\quad(\rmI\rmI\rmI)\;\{\;M_4\times 2,M_5\times 2\;\}
\] 
The MCM modules in the class of $(\rmI)$ generate the both $g_{3,1}$ and $g_{3,3}$ at the same time by $(M_3\xrightarrow{1_{M_3}}M_3)$ and 
those of $(\rmI\rmI)$ generate either $g_{3,1}$ or $g_{3,3}$. Also, those of $(\rmI\rmI\rmI)$ only generate $g_{3,3}$. 
In order to construct enough surjections, we should combine MCM modules in $(\rmI\rmI)$ and $(\rmI\rmI\rmI)$, that is, 
we use $(\rmI\rmI)$'s for $g_{3,1}$ and $(\rmI\rmI\rmI)$'s for $g_{3,3}$.
After making an appropriate pair of them $($we can make four pairs$)$, we have two remaining MCM modules in $(\rmI\rmI)$. 
We can use a one of remainders for $g_{3,1}$ and the other for $g_{3,3}$. 

Thus, the dual $F$-signature of $M_3$ can take 
$s(M_3)\le\frac{4}{12}+\frac{4}{12}+\frac{1}{12}=\frac{9}{12}$ and the following table and Lemma\,\ref{key_lemma} asserts equality 
$($in this table, we use $5\xrightarrow{E}3^2$ for generating $g_{3,3}$ $)$. 
 {\footnotesize\[
 \begin{tabular}{c|ccccccc|c} 
   $M_t$&$R$&$R$&$M_1$&$M_2$&$M_3$&$M_4$&$M_5$&Total \\ \hline
   Path&$aC$&$aCdD$&$bC$&$C$&$1_{M_3}$&$D$&$E$& \\ \hline
   $g_{3,1}$&$1$&$0$&$0$&$4$&$4$&$0$&$0$&$9$ \\
   $g_{3,2}$&$0$&$0$&$2$&$4$&$4$&$0$&$0$&$10$ \\
   $g_{3,3}$&$0$&$1$&$0$&$0$&$4$&$2$&$2$&$9$ \\
   $g_{3,4}$&$0$&$0$&$0$&$4$&$4$&$2$&$2$&$12$ \\
 \end{tabular}
 \]}
\end{itemize}


By using a similar method, we have the dual $F$-signature of other MCM modules. 
The following is the value of the dual $F$-signature corresponding to the Dynkin diagram $D_5$. 
{\scriptsize 
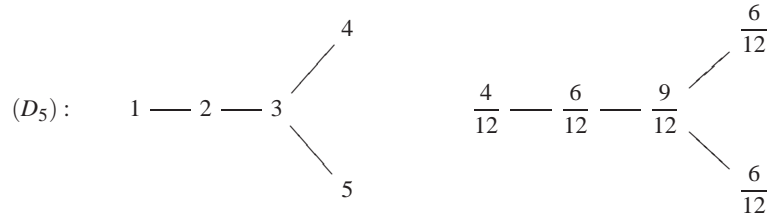
\[\xymatrix@C=15pt@R=5pt{
&&&&4& & &&&\displaystyle\frac{6}{12} \\
(D_5):&1\ar@{-}[r]&2\ar@{-}[r]&3\ar@{-}[ru]\ar@{-}[rd]& & 
&\displaystyle\frac{4}{12}\ar@{-}[r]&\displaystyle\frac{6}{12}\ar@{-}[r]&\displaystyle\frac{9}{12}\ar@{-}[ru]\ar@{-}[rd]& \\
&&&&5& & &&&\displaystyle\frac{6}{12}
} \]}
\end{ex}

Now, we move to the case of type $D_n$ while referring to the type $D_5$. 
Since the basic idea of determining the dual $F$-signature is the same as above, we only mention an outline for the case of $D_n$ 
and also for $E_6,E_7$ and $E_8$ (see subsection\,\ref{type_E6}, \ref{type_E7} and \ref{type_E8}).   

\medskip

The AR quiver of type $D_n$ is the following.
{\small 
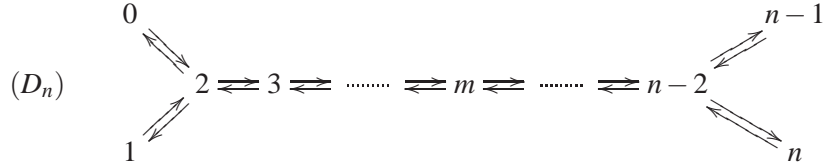
\[\xymatrix@C=15pt@R=10pt{
 &0\ar@<0.3ex>[rd]&&&&&&&&&n-1\ar@<0.3ex>[ld]\\
 (D_n)&&2\ar@<0.3ex>[r]\ar@<0.3ex>[lu]\ar@<0.3ex>[ld]&3\ar@<0.3ex>[r]\ar@<0.3ex>[l]&\ar@{..}[r]\ar@<0.3ex>[l]&\ar@<0.3ex>[r]&m\ar@<0.3ex>[l]\ar@<0.3ex>[r]&\ar@{..}[r]\ar@<0.3ex>[l]&\ar@<0.3ex>[r] &n-2\ar@<0.3ex>[l]\ar@<0.3ex>[ru]\ar@<0.3ex>[rd] \\
 &1\ar@<0.3ex>[ru]&&&&&&&&&n\ar@<0.3ex>[lu] 
}\] }

We rewrite it as a repetition of the original one $($i.e. a translation quiver $\ZZ D_n)$.
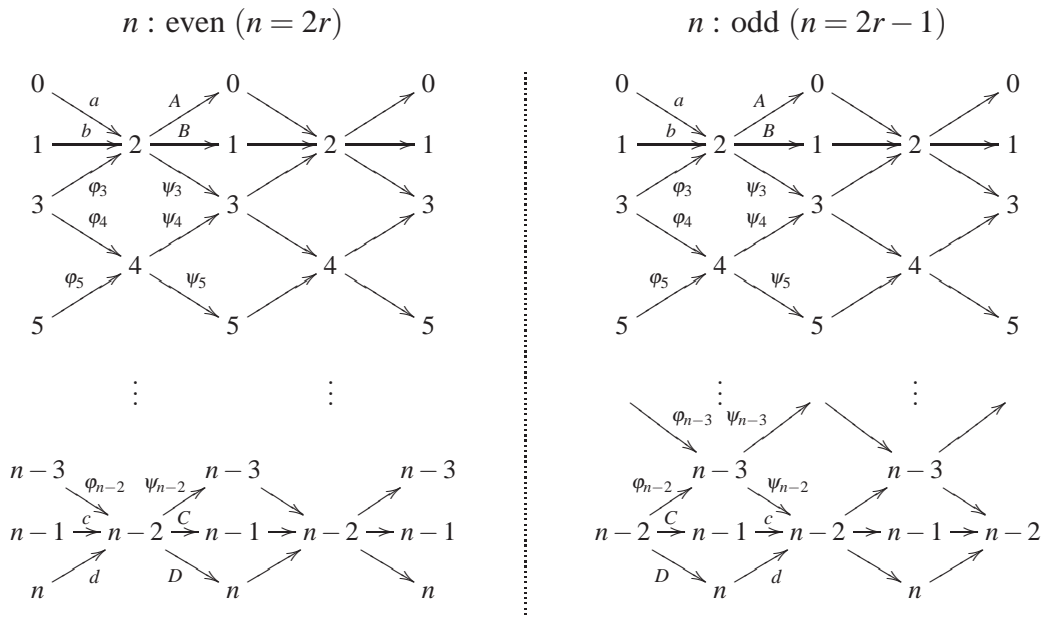
\[
 \begin{array}{ccc}
 \mbox{$n$ : even\;$(n=2r)$}&&\mbox{$n$ : odd\;$(n=2r-1)$} \\
 \begin{array}{c}
{\footnotesize 
\xymatrix@C=10pt@R=10pt{
0\ar[rd]^a&&0\ar[rd]&&0 \\
1\ar[r]^b&2\ar[ru]^A\ar[r]^B\ar[rd]_{\psi_3}&1\ar[r]&2\ar[ru]\ar[r]\ar[rd]&1  \\
3\ar[ru]_{\varphi_3}\ar[rd]^{\varphi_4}&&3\ar[ru]\ar[rd]&&3  \\
&4\ar[ru]^{\psi_4}\ar[rd]^{\psi_5}&&4\ar[ru]\ar[rd]&  \\
5\ar[ru]^{\varphi_5}&&5\ar[ru]&&5  \\
&\vdots&&\vdots&  \\
n-3\ar[rd]^{\varphi_{n-2}}&&n-3\ar[rd]&&n-3  \\
n-1\ar[r]^c&n-2\ar[ru]^{\psi_{n-2}}\ar[r]^C\ar[rd]_D&n-1\ar[r]&n-2\ar[ru]\ar[r]\ar[rd]&n-1  \\
n\ar[ru]_d&&n\ar[ru]&&n  
} }
 \end{array} &
 \begin{array}{c}
{\footnotesize 
\xymatrix@C=3pt@R=15pt{
\ar@{.}[dddddddddd] \\
\\
\\
\\
\\
\\
\\
\\
\\
\\
\\
} }
 \end{array} &
 \begin{array}{c}
{\footnotesize 
\xymatrix@C=10pt@R=10pt{
0\ar[rd]^a&&0\ar[rd]&&0 \\
1\ar[r]^b&2\ar[ru]^A\ar[r]^B\ar[rd]_{\psi_3}&1\ar[r]&2\ar[ru]\ar[r]\ar[rd]&1  \\
3\ar[ru]_{\varphi_3}\ar[rd]^{\varphi_4}&&3\ar[ru]\ar[rd]&&3  \\
&4\ar[ru]^{\psi_4}\ar[rd]^{\psi_5}&&4\ar[ru]\ar[rd]&  \\
5\ar[ru]^{\varphi_5}&&5\ar[ru]&&5  \\
\ar[rd]^{\varphi_{n-3}}&\vdots&\ar[rd]&\vdots&  \\
&n-3\ar[rd]^{\psi_{n-2}}\ar[ru]^{\psi_{n-3}}&&n-3\ar[rd]\ar[ru]&  \\
n-2\ar[ru]^{\varphi_{n-2}}\ar[r]^C\ar[rd]_D&n-1\ar[r]^c&n-2\ar[ru]\ar[r]\ar[rd]&n-1\ar[r]&n-2  \\
&n\ar[ru]_d&&n\ar[ru]&  
} }
 \end{array} 
\end{array}
\] 
Also, this quiver has the relations.
\[
\begin{array}{ccc}
 \mbox{$n$ : even\;$(n=2r)$}&&\mbox{$n$ : odd\;$(n=2r-1)$} \\
{\footnotesize
\begin{array}{c}
\begin{cases}
aA=0,& bB=0, \\
Aa+Bb+\psi_3\varphi_3=0,& cC=0, \\
\psi_{n-2}\varphi_{n-2}+Cc+Dd=0,& dD=0, \\
\varphi_{2l-1}\psi_{2l-1}+\varphi_{2l}\psi_{2l}=0 & (l=2,\cdots,r-1),\\
\psi_{2l}\varphi_{2l}+\psi_{2l+1}\varphi_{2l+1}=0 & (l=2,\cdots,r-2).
\end{cases}
\end{array}} &
 \begin{array}{c}
{\footnotesize 
\xymatrix@C=3pt@R=7pt{
\ar@{.}[dddddddd] \\
\\
\\
\\
\\
\\
\\
\\
\\
} }
 \end{array} &
{\footnotesize
 \begin{array}{c}
\begin{cases}
aA=0,& bB=0, \\
Aa+Bb+\psi_3\varphi_3=0,& cC=0, \\
\varphi_{n-2}\psi_{n-2}+Cc+Dd=0,& dD=0, \\
\varphi_{2l-1}\psi_{2l-1}+\varphi_{2l}\psi_{2l}=0 & (l=2,\cdots,r-2),\\
\psi_{2l}\varphi_{2l}+\psi_{2l+1}\varphi_{2l+1}=0 & (l=2,\cdots,r-2).
\end{cases}
 \end{array}} 
\end{array}
\] 

\bigskip

Applying the counting argument, we have the following picture.

{\footnotesize 
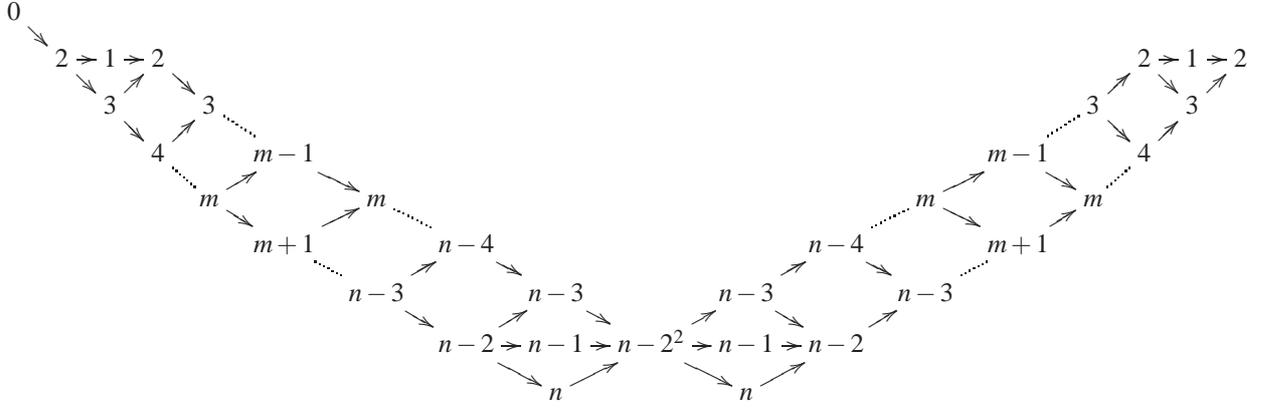
\[\xymatrix@C=7pt@R=5pt{
0\ar[rd] \\
&2\ar[rd]\ar[r]&1\ar[r]&2\ar[rd]&&&&&&&&&&&&2\ar[r]\ar[rd]&1\ar[r]&2 \\
&&3\ar[ru]\ar[rd]&&3\ar@{.}[rd]&&&&&&&&&&3\ar[ru]\ar[rd]&&3\ar[ru] \\
&&&4\ar[ru]\ar@{.}[rd]&&m-1\ar[rd]&&&&&&&&m-1\ar@{.}[ru]\ar[rd]&&4\ar[ru] \\
&&&&m\ar[ru]\ar[rd]&&m\ar@{.}[rd]&&&&&&m\ar[ru]\ar[rd]&&m\ar@{.}[ru] \\
&&&&&m+1\ar[ru]\ar@{.}[rd]&&n-4\ar[rd]&&&&n-4\ar@{.}[ru]\ar[rd]&&m+1\ar[ru] \\
&&&&&&n-3\ar[ru]\ar[rd]&&n-3\ar[rd]&&n-3\ar[ru]\ar[rd]&&n-3\ar@{.}[ru] \\
&&&&&&&n-2\ar[ru]\ar[r]\ar[rd]&n-1\ar[r]&n-2^2\ar[ru]\ar[r]\ar[rd]&n-1\ar[r]&n-2\ar[ru] \\
&&&&&&&&n\ar[ru]&&n\ar[ru] \\
} \]}

Since there is no big differences between an even number case and an odd number case, we will explain the former case. 
Thus, in the rest of this subsection, we suppose $n=2r$. 

\bigskip

\begin{itemize}
 \item[] \underline{\textbf{$\cdot$ the case of $M_1$ in $D_n$}\;}

 \medskip
 Since $\rank_RM_1=1$, the multiplicity $d_{t,1}$ of $M_t$ in ${}^eM_1$ is the following. 

{\footnotesize 
 \begin{equation}
 \label{Dn_M1_multi}
 \begin{tabular}{c|ccccccccc} 
   $M_t$&$R$&$M_1$&$M_2$&$\cdots$&$M_m$&$\cdots$&$M_{n-2}$&$M_{n-1}$&$M_n$ \\ \hline
   $d_{t,1}$&1&1&2&$\cdots$&2&$\cdots$&2&1&1 \\ 
 \end{tabular}
 \end{equation}}

 \smallskip

The paths which generate $g_{1,1}$ are only $(R\xrightarrow{aB}M_1)$, $(M_1\xrightarrow{1_{M_1}}M_1)$ and $(M_2\xrightarrow{B}M_1)\times 2$. 
Thus, the dual $F$-signature of $M_1$ can take $s(M_1)\le\frac{1}{4(n-2)}+\frac{1}{4(n-2)}+\frac{2}{4(n-2)}=\frac{4}{4(n-2)}$.
The following table and Lemma\,\ref{key_lemma} shows $s(M_1)=\displaystyle\frac{4}{4(n-2)}$.

 {\footnotesize\[
 \begin{tabular}{c|ccccccccccc|c} 
   $M_t$&$R$&$M_1$&$M_2$&$M_3$&$M_4$&$\cdots$&$M_m$&$\cdots$&$M_{n-2}$&$M_{n-1}$&$M_n$&Total \\ \hline
   Path&$aB$&$1_{M_1}$&$B$&$\varphi_3B$&0&$\cdots$&0&$\cdots$&0&0&0& \\ \hline
   $g_{1,1}$&1&1&2&0&0&$\cdots$&0&$\cdots$&0&0&0&4 \\
   $g_{1,2}$&0&1&2&2&0&$\cdots$&0&$\cdots$&0&0&0&5 \\
 \end{tabular}
 \]}

\bigskip

 \item[] \underline{\textbf{$\cdot$ the case of $M_m\;(\;2\le m\le n/2\;)$ in $D_n$}\;}

 \medskip

 Since $\rank_RM_m=2$, the multiplicity $d_{t,m}$ of $M_t$ in ${}^eM_m$ is the following.

{\footnotesize 
 \begin{equation}
 \label{Dn_Mm_multi}
 \begin{tabular}{c|ccccccccc} 
   $M_t$&$R$&$M_1$&$M_2$&$\cdots$&$M_m$&$\cdots$&$M_{n-2}$&$M_{n-1}$&$M_n$ \\ \hline
   $d_{t,m}$&2&2&4&$\cdots$&4&$\cdots$&4&2&2 \\ 
 \end{tabular}
 \end{equation}}

 \smallskip

In the same way as the previous example, we can see that the MCM $R$-modules $R\times 2, M_2\times 4, \cdots, M_m\times 4$ can generate $g_{m,1}$, 
and we have no other such MCMs. 
Thus, the dual $F$-signature of $M_m$ can take $s(M_m)\le\frac{2}{4(n-2)}+\frac{4(m-1)}{4(n-2)}=\frac{4m-2}{4(n-2)}$. 
By Lemma\,\ref{key_lemma} and the following table, we conclude $s(M_m)=\displaystyle\frac{4m-2}{4(n-2)}$. 

 {\footnotesize\[
 \begin{tabular}{c|ccccccccccc|c} 
   $M_t$&$R$&$M_1$&$M_2$&$\cdots$&$M_{m-1}$&$M_m$&$M_{m+1}$&$\cdots$&$M_{n-2}$&$M_{n-1}$&$M_n$&Total \\ \hline
   Path&$a\Gamma^2_m$&$b\Gamma^2_m$&$\Gamma^2_m$&&$\varphi_m$&$1_{M_m}$&$\varphi_{m+1}$&&$\Lambda^{n-2}_m$&$c\Lambda^{n-2}_m$&0& \\ \hline
   $g_{m,1}$&2&0&4&$\cdots$&4&4&0&$\cdots$&0&0&0&$4m-2$ \\
   $g_{m,2}$&0&2&4&$\cdots$&4&4&0&$\cdots$&0&0&0&$4m-2$ \\
   $g_{m,3}$&0&0&0&$\cdots$&0&4&4&$\cdots$&4&2&0&$4n-4m-2$ \\
   $g_{m,4}$&0&0&0&$\cdots$&0&4&4&$\cdots$&4&2&0&$4n-4m-2$ \\
 \end{tabular}
 \]}

Here, we set $\Gamma^i_m\coloneqq \psi_{i+1}\varphi_{i+2}\cdots\psi_{m-1}\varphi_m,\; \Lambda^i_m\coloneqq \psi_{i}\varphi_{i-1}\cdots\psi_{m+2}\varphi_{m+1}$. 
In this table, we suppose that $m$ is an even number. 
Although the notation is slightly different, we obtain a similar table for an odd number case.  

\bigskip

 \item[] \underline{\textbf{$\cdot$ the case of $M_m\;(\;n/2< m\le n-2\;)$ in $D_n$}\;}

 \medskip

The multiplicity $d_{t,m}$ of $M_t$ in ${}^eM_m$ is the same as the table (\ref{Dn_Mm_multi}). 
In order to obtain the upper bounds of $s(M_m)$, we classify the MCM $R$-modules in ${}^eM_m$ as follows;  
\[
(\rmI)\;\{\;M_m\times 4\;\}\quad(\rmI\rmI)\;\{\;R\times 2, M_2\times 4, \cdots, M_{m-1}\times 4\;\}
\]
\[
(\rmI\rmI\rmI)\;\{\;M_{m+1}\times 4, \cdots, M_{n-2}\times 4, M_{n-1}\times 2, M_n\times 2\;\}
\]
where the class\,$(\rmI)$ $($resp. $(\rmI\rmI),(\rmI\rmI\rmI))$ is the set of MCM $R$-modules which generate $g_{m,1}$ and $g_{m,3}$ at the same time
$($resp. either $g_{m,1}$ or $g_{m,3}$, only $g_{m,3}$ $)$. 
In order to construct enough surjections, we should combine MCM modules in $(\rmI\rmI)$ and $(\rmI\rmI\rmI)$, that is, 
we use $(\rmI\rmI)$'s for $g_{m,1}$ and $(\rmI\rmI\rmI)$'s for $g_{m,3}$.
After making an appropriate pair of them (we obtain $4(n-m-1) pairs$), we have $2(4m-2n-1)$ remaining MCM modules in $(\rmI\rmI)$. 
We can use a half of remainders for $g_{m,1}$ and the others for $g_{m,3}$. 
Thus, we have the upper bounds $s(M_m)\le\frac{4}{4(n-2)}+\frac{4(n-m-1)}{4(n-2)}+\frac{4m-2n-1}{4(n-2)}=\frac{2n-1}{4(n-2)}$, 
and we have the following table. $($In this table, we suppose that $m$ is an even number. 
Although the notation is slightly different, we obtain a similar table for an odd number case.$)$ 

 {\footnotesize\[
 \begin{tabular}{c|ccccccccccccc} 
   $M_t$&$R$&$R$&$M_1$&$M_1$&$M_2$&$\cdots$&$M_{2m-n}$&$M_2$&$\cdots$&$M_{2m-n}$\\ \hline
   Path&$a\Gamma^2_m$&$a\Gamma^2_{n-2}Dd\Lambda^{n-2}_m$&$b\Gamma^2_m$&$b\Gamma^2_{n-2}Dd\Lambda^{n-2}_m$&$\Gamma^2_m$&&$\Gamma^{2m-n}_m$&$\Gamma^2_{n-2}Dd\Lambda^{n-2}_m$&&$\Gamma^{2m-n}_{n-2}Dd\Lambda^{n-2}_m$ \\ \hline
   $g_{m,1}$&1&0&0&0&2&$\cdots$&2&0&$\cdots$&0 \\
   $g_{m,2}$&0&0&1&0&2&$\cdots$&2&0&$\cdots$&0 \\
   $g_{m,3}$&0&1&0&0&0&$\cdots$&0&2&$\cdots$&2 \\
   $g_{m,4}$&0&0&0&1&0&$\cdots$&0&2&$\cdots$&2 \\
 \end{tabular}
\]
\[
 \begin{tabular}{ccccccccc|c} 
   $M_{2m-n+1}$&$\cdots$&$M_{m-1}$&$M_m$&$M_{m+1}$&$\cdots$&$M_{n-2}$&$M_{n-1}$&$M_n$&Total \\ \hline
   $\varphi_{2m-n+2}\Gamma^{2m-n+2}_m$&&$\varphi_m$&$1_{M_m}$&$\varphi_{m+2}\Gamma^{m+2}_{n-2}Dd\Lambda^{n-2}_m$&&$\Lambda^{n-2}_m$&$c\Lambda^{n-2}_m$&$d\Lambda^{n-2}_m$& \\ \hline
   4&$\cdots$&4&4&0&$\cdots$&0&0&0&$2n-1$ \\
   4&$\cdots$&4&4&0&$\cdots$&0&0&0&$2n-1$ \\
   0&$\cdots$&0&4&4&$\cdots$&4&2&2&$2n-1$ \\
   0&$\cdots$&0&4&4&$\cdots$&4&2&2&$2n-1$ \\
 \end{tabular}
 \]}

Thus, we conclude $s(M_m)=\displaystyle\frac{2n-1}{4(n-2)}$ by Lemma\,\ref{key_lemma}.

\bigskip

 \item[] \underline{\textbf{$\cdot$ the case of $M_{n-1}$ in $D_n$}\;}

 \medskip

The multiplicity $d_{t,n-1}$ of $M_t$ in ${}^eM_{n-1}$ is the same as the table (\ref{Dn_M1_multi}). 
Similarly, we have the upper bounds $s(M_{n-1})\le\frac{2(n-2)}{4(n-2)}$ by selecting paths which generate $g_{n-1,1}$, 
and we have the following table. $($In this table, we suppose that $m$ is an even number, the same as the previous case.$)$ 

 {\footnotesize\[
 \begin{tabular}{c|ccccccccc|c} 
   $M_t$&$R$&$M_1$&$M_2$&$\cdots$&$M_m$&$\cdots$&$M_{n-2}$&$M_{n-1}$&$M_n$&Total \\ \hline
   Path&$a\Gamma^2_{n-2}C$&$b\Gamma^2_{n-2}C$&$\Gamma^2_{n-2}C$&&$\Gamma^m_{n-2}C$&&$C$&$1_{M_{n-1}}$&0& \\ \hline
   $g_{n-1,1}$&1&0&2&$\cdots$&2&$\cdots$&2&1&0&2(n-2) \\
   $g_{n-1,2}$&0&1&2&$\cdots$&2&$\cdots$&2&1&0&2(n-2) \\
 \end{tabular}
 \]}

Thus, we conclude $s(M_{n-1})=\displaystyle\frac{2(n-2)}{4(n-2)}$ by Lemma\,\ref{key_lemma}.

\bigskip

 \item[] \underline{\textbf{$\cdot$ the case of $M_n$ in $D_n$}\;}

 \medskip

The AR quiver of $D_n$ is symmetric with respect to $M_{n-1}$ and $M_n$, and $\rank_RM_{n-1}=\rank_RM_n$. 
So we have $s(M_n)=\displaystyle\frac{2(n-2)}{4(n-2)}$ in the same way.

\end{itemize}

\subsection{Type $E_6$} 
\label{type_E6}

The AR quiver of type $E_6$ $($as the form of $\ZZ E_6)$ is 

{\footnotesize 
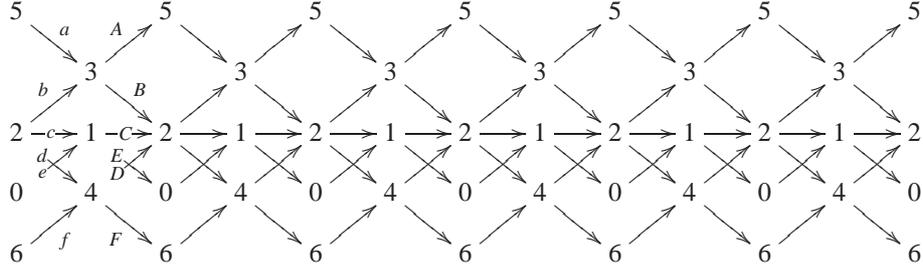
\[\xymatrix@C=17pt@R=10pt{
 5\ar[rd]^a&&5\ar[rd]&&5\ar[rd]&&5\ar[rd]&&5\ar[rd]&&5\ar[rd]&&5 \\
 &3\ar[ru]^A\ar[rd]^B&&3\ar[ru]\ar[rd]&&3\ar[ru]\ar[rd]&&3\ar[ru]\ar[rd]&&3\ar[ru]\ar[rd]&&3\ar[ru]\ar[rd]& \\
 2\ar[ru]^b\ar[r]|(0.45){\,c}\ar[rd]|(0.35){d}&1\ar[r]|(0.45){\,C}\ar[rd]|(0.35)E&2\ar[ru]\ar[r]\ar[rd]&1\ar[r]\ar[rd]&2\ar[ru]\ar[r]\ar[rd]&1\ar[r]\ar[rd]&2\ar[ru]\ar[r]\ar[rd]&1\ar[r]\ar[rd]&2\ar[ru]\ar[r]\ar[rd]&1\ar[r]\ar[rd]&2\ar[ru]\ar[r]\ar[rd]&1\ar[r]\ar[rd]&2 \\
 0\ar[ru]|(0.35){\,e\,}&4\ar[ru]|(0.35)D\ar[rd]_F&0\ar[ru]&4\ar[ru]\ar[rd]&0\ar[ru]&4\ar[ru]\ar[rd]&0\ar[ru]&4\ar[ru]\ar[rd]&0\ar[ru]&4\ar[ru]\ar[rd]&0\ar[ru]&4\ar[ru]\ar[rd]&0 \\
 6\ar[ru]_f&&6\ar[ru]&&6\ar[ru]&&6\ar[ru]&&6\ar[ru]&&6\ar[ru]&&6 
}\] }

with relations 
\[
\begin{cases}
aA=0,& bB+cC+dD=0, \\
eE=0,& fF=0, \\
Aa+Bb=0,& Cc+Ee=0, \\
Dd+Ff=0.
\end{cases}
\]

After applying the counting argument, we have the following quiver. 

{\footnotesize
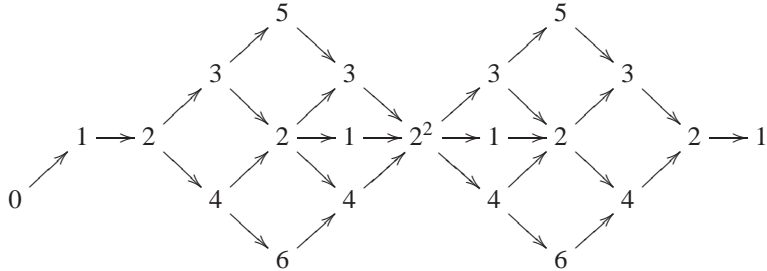
\[\xymatrix@C=14pt@R=10pt{
 &&&&5\ar[rd]&&&&5\ar[rd]&&& \\
 &&&3\ar[ru]\ar[rd]&&3\ar[rd]&&3\ar[ru]\ar[rd]&&3\ar[rd]&& \\
 &1\ar[r]&2\ar[ru]\ar[rd]&&2\ar[ru]\ar[r]\ar[rd]&1\ar[r]&2^2\ar[ru]\ar[r]\ar[rd]&1\ar[r]&2\ar[ru]\ar[rd]&&2\ar[r]&1 \\
 0\ar[ru]&&&4\ar[ru]\ar[rd]&&4\ar[ru]&&4\ar[ru]\ar[rd]&&4\ar[ru]&& \\
 &&&&6\ar[ru]&&&&6\ar[ru]&&& 
}\] }

\begin{itemize}
 \item[] \underline{\textbf{$\cdot$ the case of $M_1$ in $E_6$}\;}

 \medskip
 Since $\rank_RM_1=2$, the multiplicity $d_{t,1}$ of $M_t$ in ${}^eM_1$ is the following.
 
{\footnotesize
 \begin{equation}
 \label{E6_M1_multi}
 \begin{tabular}{c|ccccccc} 
   $M_t$&$R$&$M_1$&$M_2$&$M_3$&$M_4$&$M_5$&$M_6$ \\ \hline
   $d_{t,1}$&2&4&6&4&4&2&2 \\ 
 \end{tabular}
 \end{equation}}

 \smallskip

Since the paths which generate $g_{1,1}$ are only $(R\xrightarrow{e}M_1)\times 2$ and $(M_1\xrightarrow{1_{M_1}}M_1)\times 4$, 
the dual $F$-signature of $M_1$ can take $s(M_1)\le\frac{6}{24}$. 
The following table and Lemma\,\ref{key_lemma} assert $s(M_1)=\displaystyle\frac{6}{24}$.
 {\footnotesize\[
 \begin{tabular}{c|ccccccc|c} 
   $M_t$&$R$&$M_1$&$M_2$&$M_3$&$M_4$&$M_5$&$M_6$&Total \\ \hline
   Path&$e$&$1_{M_1}$&c&0&0&0&0& \\ \hline
   $g_{1,1}$&2&4&0&0&0&0&0&6 \\
   $g_{1,2}$&0&4&6&0&0&0&0&10 \\
   $g_{1,3}$&0&4&6&0&0&0&0&10 \\
   $g_{1,4}$&0&4&6&0&0&0&0&10 
 \end{tabular}
 \]}

\bigskip

 \item[] \underline{\textbf{$\cdot$ the case of $M_2$ in $E_6$}\;}

 \medskip

 Since $\rank_RM_2=3$, the multiplicity $d_{t,2}$ of $M_t$ in ${}^eM_2$ is the following.

{\footnotesize
 \begin{equation}
 \begin{tabular}{c|ccccccc} 
   $M_t$&$R$&$M_1$&$M_2$&$M_3$&$M_4$&$M_5$&$M_6$ \\ \hline
   $d_{t,2}$&3&6&9&6&6&3&3 \\ 
 \end{tabular}
 \end{equation}}

Similarly, we have the upper bounds $s(M_2)\le\frac{18}{24}$ by selecting paths which generate $g_{2,1}$. 
Suppose that $g_{2,3}$ $($resp. $g_{2,4})$ is generated through a path which factor through $1\xrightarrow{C}2^2$ 
$($resp. $4\xrightarrow{D}2^2)$. So we can use paths which factor through $3\xrightarrow{B}2^2$ for either $g_{2,3}$ or $g_{2,4}$ 
$($see the arguments in the case of $M_3$ in $D_5)$. 
Then we have the following table.

 {\footnotesize\[
 \begin{tabular}{c|ccccccc|c} 
   $M_t$&$R$&$M_1$&$M_2$&$M_3$&$M_4$&$M_5$&$M_6$&Total \\ \hline
   Path&$eC$&$C$&$1_{M_2}$&$B$&$D$&0&$fD$& \\ \hline
   $g_{2,1}$&3&6&9&0&0&0&0&18 \\
   $g_{2,2}$&0&0&9&6&6&0&0&21 \\
   $g_{2,3}$&0&6&9&6&0&0&0&21 \\
   $g_{2,4}$&0&0&9&0&6&0&3&18 \\ 
   $g_{2,5}$&0&6&9&6&6&0&0&27 \\ 
   $g_{2,6}$&0&0&9&6&6&0&3&24  
 \end{tabular}
 \]}

Thus, we conclude $s(M_2)=\displaystyle\frac{18}{24}$ by Lemma\,\ref{key_lemma}.

\bigskip

 \item[] \underline{\textbf{$\cdot$ the case of $M_3$ in $E_6$}\;}

 \medskip

Since $\rank_RM_3=2$, the multiplicity $d_{t,3}$ of $M_t$ in ${}^eM_3$ is the same as the table (\ref{E6_M1_multi}) 
and we have the upper bounds $s(M_3)\le\frac{16}{24}$ by selecting paths which generate $g_{3,1}$.
The equality follows from the following table and Lemma\,\ref{key_lemma}.

 {\footnotesize\[
 \begin{tabular}{c|ccccccc|c} 
   $M_t$&$R$&$M_1$&$M_2$&$M_3$&$M_4$&$M_5$&$M_6$&Total \\ \hline
   Path&$eCb$&$Cb$&$b$&$1_{M_3}$&$Db$&$a$&$fDb$& \\ \hline
   $g_{3,1}$&2&4&6&4&0&0&0&16 \\
   $g_{3,2}$&0&0&6&4&4&2&0&16 \\
   $g_{3,3}$&0&4&6&4&4&0&2&20 \\
   $g_{3,4}$&0&4&6&4&4&2&0&20   
 \end{tabular}
 \]}

\bigskip

 \item[] \underline{\textbf{$\cdot$ the case of $M_4$ in $E_6$}\;}

 \medskip

 The AR quiver of $E_6$ is symmetric with respect to $M_3$ and $M_4$, and $\rank_RM_3=\rank_RM_4$.
 So we have $s(M_4)=\displaystyle\frac{16}{24}$ in the same way.

\bigskip

 \item[] \underline{\textbf{$\cdot$ the case of $M_5$ in $E_6$}\;}

 \medskip

 Since $\rank_RM_5=1$, the multiplicity $d_{t,5}$ of $M_t$ in ${}^eM_5$ is the following.

{\footnotesize
 \begin{equation}
 \begin{tabular}{c|ccccccc} 
   $M_t$&$R$&$M_1$&$M_2$&$M_3$&$M_4$&$M_5$&$M_6$ \\ \hline
   $d_{t,5}$&1&2&3&2&2&1&1 \\ 
 \end{tabular}
 \end{equation}}

 \smallskip

We have the upper bounds $s(M_5)\le\frac{9}{24}$ by selecting paths which generate $g_{5,1}$, 
and the following table and Lemma\,\ref{key_lemma} assert the equality.

 {\footnotesize\[
 \begin{tabular}{c|ccccccc|c} 
   $M_t$&$R$&$M_1$&$M_2$&$M_3$&$M_4$&$M_5$&$M_6$&Total \\ \hline
   Path&$eCbA$&$CbA$&$bA$&$A$&$DbA$&$1_{M_5}$&0& \\ \hline
   $g_{5,1}$&1&2&3&2&0&1&0&9 \\
   $g_{5,2}$&0&2&3&2&2&1&0&10 \\  
 \end{tabular}
 \]}

\bigskip

 \item[] \underline{\textbf{$\cdot$ the case of $M_6$ in $E_6$}\;}

 \medskip

 The AR quiver of $E_6$ is symmetric with respect to $M_5$ and $M_6$, and $\rank_RM_5=\rank_RM_6$. 
 So we have $s(M_6)=\displaystyle\frac{9}{24}$ in the same way.

\bigskip

\end{itemize}

\subsection{Type $E_7$} 
\label{type_E7}

The AR quiver of type $E_7$ $($as the form of $\ZZ E_7)$ is 

{\footnotesize 
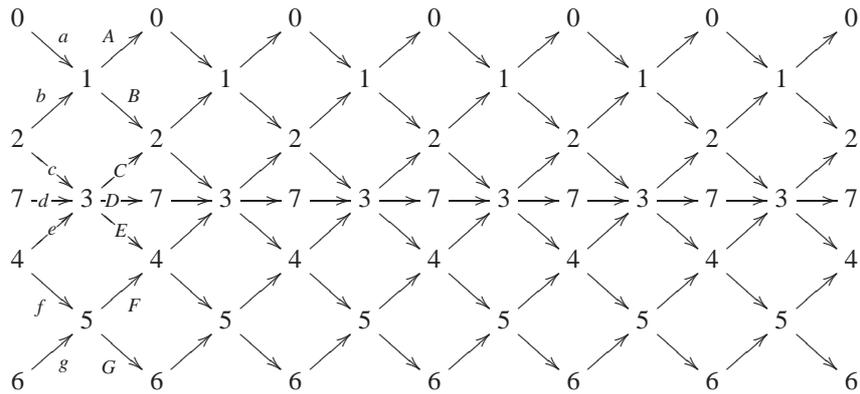
\[\xymatrix@C=15pt@R=10pt{
  0\ar[rd]^a&&0\ar[rd]&&0\ar[rd]&&0\ar[rd]&&0\ar[rd]&&0\ar[rd]&&0 \\
  &1\ar[ru]^A\ar[rd]^B&&1\ar[ru]\ar[rd]&&1\ar[ru]\ar[rd]&&1\ar[ru]\ar[rd]&&1\ar[ru]\ar[rd]&&1\ar[ru]\ar[rd]&\\
  2\ar[ru]^b\ar[rd]|{\,c\,}&&2\ar[ru]\ar[rd]&&2\ar[ru]\ar[rd]&&2\ar[ru]\ar[rd]&&2\ar[ru]\ar[rd]&&2\ar[ru]\ar[rd]&&2 \\
  7\ar[r]|(0.38){\,d\,}&3\ar[ru]|{C\,}\ar[r]|(0.38){D\,}\ar[rd]|E&7\ar[r]&3\ar[ru]\ar[r]\ar[rd]&7\ar[r]&3\ar[ru]\ar[r]\ar[rd]&7\ar[r]&3\ar[ru]\ar[r]\ar[rd]&7\ar[r]&3\ar[ru]\ar[r]\ar[rd]&
7\ar[r]&3\ar[ru]\ar[r]\ar[rd]&7 \\
  4\ar[ru]|{\,e\,}\ar[rd]_f&&4\ar[ru]\ar[rd]&&4\ar[ru]\ar[rd]&&4\ar[ru]\ar[rd]&&4\ar[ru]\ar[rd]&&4\ar[ru]\ar[rd]&&4 \\
  &5\ar[ru]_F\ar[rd]_G&&5\ar[ru]\ar[rd]&&5\ar[ru]\ar[rd]&&5\ar[ru]\ar[rd]&&5\ar[ru]\ar[rd]&&5\ar[ru]\ar[rd]& \\
  6\ar[ru]_g&&6\ar[ru]&&6\ar[ru]&&6\ar[ru]&&6\ar[ru]&&6\ar[ru]&&6 
}\] }

with relations 
\[
\begin{cases}
aA=0,& bB+cC=0, \\
dD=0,& eE+fF=0, \\
gG=0,& Aa+Bb=0, \\
Cc+Dd+Ee=0,& Ff+Gg=0.
\end{cases}
\]

After applying the counting argument, we have the following quiver. 

{\footnotesize
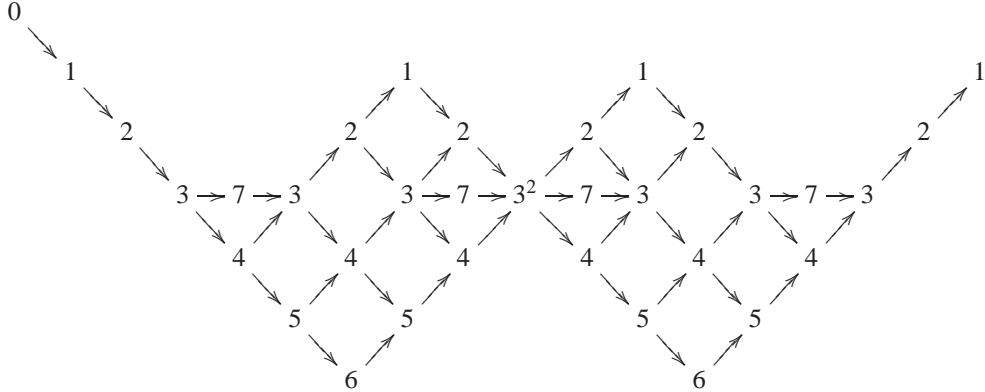
\[\xymatrix@C=10pt@R=10pt{
  0\ar[rd]&&&&&&&&&&&&&&&&&&&& \\
  &1\ar[rd]&&&&&&1\ar[rd]&&&&1\ar[rd]&&&&&&1&&&\\
  &&2\ar[rd]&&&&2\ar[ru]\ar[rd]&&2\ar[rd]&&2\ar[ru]\ar[rd]&&2\ar[rd]&&&&2\ar[ru]&&&& \\
  &&&3\ar[r]\ar[rd]&7\ar[r]&3\ar[ru]\ar[rd]&&3\ar[ru]\ar[r]\ar[rd]&7\ar[r]&3^2\ar[ru]\ar[r]\ar[rd]&
7\ar[r]&3\ar[ru]\ar[rd]&&3\ar[r]\ar[rd]&7\ar[r]&3\ar[ru]&&&&& \\
  &&&&4\ar[ru]\ar[rd]&&4\ar[ru]\ar[rd]&&4\ar[ru]&&4\ar[ru]\ar[rd]&&4\ar[ru]\ar[rd]&&4\ar[ru]&&&&&& \\
  &&&&&5\ar[ru]\ar[rd]&&5\ar[ru]&&&&5\ar[ru]\ar[rd]&&5\ar[ru]&&&&&&& \\
  &&&&&&6\ar[ru]&&&&&&6\ar[ru]&&&&&&&& 
}\] }

\begin{itemize}
 \item[] \underline{\textbf{$\cdot$ the case of $M_1$ in $E_7$}\;}

 \medskip

 Since $\rank_RM_1=2$, the multiplicity $d_{t,1}$ of $M_t$ in ${}^eM_1$ is the following.

{\footnotesize
 \begin{equation}
 \label{E7_M1_multi}
 \begin{tabular}{c|cccccccc} 
   $M_t$&$R$&$M_1$&$M_2$&$M_3$&$M_4$&$M_5$&$M_6$&$M_7$ \\ \hline
   $d_{t,1}$&2&4&6&8&6&4&2&4 \\ 
 \end{tabular}
 \end{equation}}

 \smallskip

Since the paths which generate $g_{1,1}$ are only $(R\xrightarrow{a}M_1)\times 2, (M_1\xrightarrow{1_{M_1}}M_1)\times 4$, 
we have the upper bounds $s(M_1)\le\frac{6}{48}$.

 {\footnotesize\[
 \begin{tabular}{c|cccccccc|c} 
   $M_t$&$R$&$M_1$&$M_2$&$M_3$&$M_4$&$M_5$&$M_6$&$M_7$&Total \\ \hline
   Path&$a$&$1_{M_1}$&$b$&0&0&0&0&0& \\ \hline
   $g_{1,1}$&2&4&0&0&0&0&0&0&6 \\
   $g_{1,2}$&0&4&6&0&0&0&0&0&10 \\  
   $g_{1,3}$&0&4&6&0&0&0&0&0&10 \\ 
   $g_{1,4}$&0&4&6&0&0&0&0&0&10 \\ 
 \end{tabular}
 \]}

We conclude $s(M_1)=\displaystyle\frac{6}{48}$ by Lemma\,\ref{key_lemma} and the above table.

\bigskip

 \item[] \underline{\textbf{$\cdot$ the case of $M_2$ in $E_7$}\;}

 \medskip

 Since $\rank_RM_2=3$, the multiplicity $d_{t,2}$ of $M_t$ in ${}^eM_2$ is the following.

{\footnotesize
 \begin{equation}
 \label{E7_M2_multi}
 \begin{tabular}{c|cccccccc} 
   $M_t$&$R$&$M_1$&$M_2$&$M_3$&$M_4$&$M_5$&$M_6$&$M_7$ \\ \hline
   $d_{t,2}$&3&6&9&12&9&6&3&6 \\ 
 \end{tabular}
 \end{equation}}

 \smallskip

Similarly, we have the upper bounds $s(M_2)\le\frac{18}{48}$ by selecting paths which generate $g_{2,1}$. 
The following table and Lemma\,\ref{key_lemma} assert the equality. 

 {\footnotesize\[
 \begin{tabular}{c|cccccccc|c} 
   $M_t$&$R$&$M_1$&$M_2$&$M_3$&$M_4$&$M_5$&$M_6$&$M_7$&Total \\ \hline
   Path&$aB$&$B$&$1_{M_2}$&0&$eC$&0&0&0& \\ \hline
   $g_{2,1}$&3&6&9&0&0&0&0&0&18 \\
   $g_{2,2}$&0&0&9&0&9&0&0&0&18 \\  
   $g_{2,3}$&0&6&9&0&9&0&0&0&24 \\ 
   $g_{2,4}$&0&0&9&0&9&0&0&0&18 \\ 
   $g_{2,5}$&0&6&9&0&9&0&0&0&24 \\ 
   $g_{2,6}$&0&0&9&0&9&0&0&0&18  
 \end{tabular}
 \]}

\bigskip

 \item[] \underline{\textbf{$\cdot$ the case of $M_3$ in $E_7$}\;}

 \medskip

 Since $\rank_RM_3=4$, the multiplicity $d_{t,3}$ of $M_t$ in ${}^eM_3$ is the following.

{\footnotesize
 \begin{equation}
 \begin{tabular}{c|cccccccc} 
   $M_t$&$R$&$M_1$&$M_2$&$M_3$&$M_4$&$M_5$&$M_6$&$M_7$ \\ \hline
   $d_{t,3}$&4&8&12&16&12&8&4&8 \\ 
 \end{tabular}
 \end{equation}}

In this case, we need to pay attention for determining the upper bounds. 
The MCM $R$-module $M_3$ can generate both $g_{3,1}$ and $g_{3,2}$ through the path $(M_3\xrightarrow{1_{M_3}}M_3)$. 
Similarly, we can read off that $R, M_1, M_2$ generate either $g_{3,1}$ or $g_{3,2}$ and $M_4, M_7$ generate $g_{3,2}$ 
but don't generate $g_{3,1}$. 
Collectively, we classify the MCM $R$-modules in ${}^eM_3$ as follows;  
\[
(\rmI)\;\{\;M_3\times 16\;\}\quad(\rmI\rmI)\;\{\;R\times 4, M_1\times 8, M_2\times 12\;\}\quad(\rmI\rmI\rmI)\;\{\;M_4\times 12,M_7\times 8\;\}
\]
where the class$(\rmI)$ $($resp. $(\rmI\rmI),(\rmI\rmI\rmI))$ is the set of MCM $R$-modules which generate $g_{3,1}$ and $g_{3,2}$ at the same time 
$($resp. either $g_{3,1}$ or $g_{3,2}$, only $g_{3,2}$ $)$. 
In order to construct enough surjections, we should combine MCM modules in $(\rmI\rmI)$ and $(\rmI\rmI\rmI)$, that is, 
we use $(\rmI\rmI)$'s for $g_{3,1}$ and $(\rmI\rmI\rmI)$'s for $g_{3,2}$.
After making an appropriate pair of them, we have four remaining MCM modules in $(\rmI\rmI)$. 
We can use half of remainders for $g_{3,1}$ and others for $g_{3,2}$. 
Thus, we have the upper bounds $s(M_3)\le\frac{16}{48}+\frac{20}{48}+\frac{2}{48}=\frac{38}{48}$. 

 {\footnotesize\[
 \begin{tabular}{c|ccccccccc|c} 
   $M_t$&$R$&$M_1$&$M_2$&$M_2$&$M_3$&$M_4$&$M_5$&$M_6$&$M_7$&Total \\ \hline
   Path&$aBc$&$Bc$&$c$&$cDd$&$1_{M_3}$&$e$&$Fe$&$gFe$&$d$& \\ \hline
   $g_{3,1}$&4&8&10&0&16&0&0&0&0&38 \\
   $g_{3,2}$&0&0&0&2&16&12&0&0&8&38 \\  
   $g_{3,3}$&0&0&10&0&16&12&8&0&0&46 \\ 
   $g_{3,4}$&0&8&10&0&16&0&8&0&0&42 \\ 
   $g_{3,5}$&0&0&0&2&16&12&0&4&8&42 \\ 
   $g_{3,6}$&0&0&10&2&16&12&0&0&8&48 \\  
   $g_{3,7}$&0&8&10&0&16&12&8&0&0&54 \\
   $g_{3,8}$&0&0&0&2&16&12&8&4&8&50 
 \end{tabular}
 \]}

By the above table and Lemma\,\ref{key_lemma}, we conclude $s(M_3)=\displaystyle\frac{38}{48}$. 
In this table, we fix that $g_{3,4}\,($resp. $g_{3,5})$ is a minimal generator identified with a path 
which factor through $2\xrightarrow{c}3^2\,($resp. $7\xrightarrow{d}3^2)$, 
and we can use paths which factor through $4\xrightarrow{e}3^2$ for generating either $g_{3,4}$ or $g_{3,5}$.

\bigskip

 \item[] \underline{\textbf{$\cdot$ the case of $M_4$ in $E_7$}\;}

 \medskip

 Since $\rank_RM_4=3$, the multiplicity $d_{t,4}$ of $M_t$ in ${}^eM_4$ is the same as the table (\ref{E7_M2_multi}).
In the same way as $M_3$, we classify the MCM $R$-modules in ${}^eM_4$ as 
\[
(\rmI)\;\{\;M_3\times 12, M_4\times 9\;\}\quad(\rmI\rmI)\;\{\;R\times 3, M_1\times 6, M_2\times 9\;\}\quad(\rmI\rmI\rmI)\;\{\;M_5\times 6,M_7\times 6\;\}
\]
where the class$(\rmI)$ $($resp. $(\rmI\rmI),(\rmI\rmI\rmI))$ is the set of MCM $R$-modules which generate $g_{4,1}$ and $g_{4,2}$ at the same time 
$($resp. either $g_{4,1}$ or $g_{4,2}$, only $g_{4,2}$ $)$ and obtain the upper bound $s(M_4)\le\frac{21}{48}+\frac{12}{48}+\frac{3}{48}=\frac{36}{48}$ 
in the same way as the case of $M_3$. 

 {\footnotesize\[
 \begin{tabular}{c|cccccccc|c} 
   $M_t$&$R$&$M_1$&$M_2$&$M_3$&$M_4$&$M_5$&$M_6$&$M_7$&Total \\ \hline
   Path&$aBcDdE$&$BcE$&$cE$&$E$&$1_{M_4}$&$F$&0&$dE$& \\ \hline
   $g_{4,1}$&0&6&9&12&9&0&0&0&36 \\
   $g_{4,2}$&3&0&0&12&9&6&0&6&36 \\  
   $g_{4,3}$&0&0&9&12&9&6&0&0&36 \\ 
   $g_{4,4}$&0&6&9&12&9&0&0&6&42 \\ 
   $g_{4,5}$&0&0&9&12&9&6&0&6&42 \\ 
   $g_{4,6}$&0&6&9&12&9&6&0&0&42  
 \end{tabular}
 \]}

Thus, we conclude $s(M_4)=\displaystyle\frac{36}{48}$ by Lemma\,\ref{key_lemma} and the above table. 

\bigskip

 \item[] \underline{\textbf{$\cdot$ the case of $M_5$ in $E_7$}\;}

 \medskip

 Since $\rank_RM_5=2$, the multiplicity $d_{t,5}$ of $M_t$ in ${}^eM_5$ is the same as the table (\ref{E7_M1_multi}).
In the same as before, we can classify the MCM $R$-modules in ${}^eM_5$ as 
\[
(\rmI)\;\{\;M_3\times 8, M_4\times 6, M_5\times 4\;\}\quad(\rmI\rmI)\;\{\;R\times 2, M_1\times 4, M_2\times 6\;\}\quad(\rmI\rmI\rmI)\;\{\;M_6\times 2,M_7\times 4\;\}
\]
where the class$(\rmI)$ $($resp. $(\rmI\rmI),(\rmI\rmI\rmI))$ is the set of MCM $R$-modules which generate $g_{5,1}$ and $g_{5,2}$ at the same time 
$($resp. either $g_{5,1}$ or $g_{5,2}$, only $g_{5,2}$ $)$ and obtain the upper bound $s(M_5)\le\frac{18}{48}+\frac{6}{48}+\frac{3}{48}=\frac{27}{48}$. 
By Lemma\,\ref{key_lemma} and the following table, we have $s(M_5)=\displaystyle\frac{27}{48}$. 

 {\footnotesize\[
 \begin{tabular}{c|ccccccccc|c} 
   $M_t$&$R$&$M_1$&$M_2$&$M_2$&$M_3$&$M_4$&$M_5$&$M_6$&$M_7$&Total \\ \hline
   Path&$aBcEf$&$BcEf$&$cEf$&$cDdEf$&$Ef$&$f$&$1_{M_5}$&$g$&$dEf$& \\ \hline
   $g_{5,1}$&2&4&3&0&8&6&4&0&0&27 \\
   $g_{5,2}$&0&0&0&3&8&6&4&2&4&27 \\  
   $g_{5,3}$&0&4&3&3&8&6&4&0&4&32 \\ 
   $g_{5,4}$&0&0&3&3&8&6&4&2&4&30 \\ 
 \end{tabular}
 \]}

\bigskip

 \item[] \underline{\textbf{$\cdot$ the case of $M_6$ in $E_7$}\;}

 \medskip

 Since $\rank_RM_6=1$, the multiplicity $d_{t,6}$ of $M_t$ in ${}^eM_6$ is the following.

{\footnotesize
 \begin{equation}
 \begin{tabular}{c|cccccccc} 
   $M_t$&$R$&$M_1$&$M_2$&$M_3$&$M_4$&$M_5$&$M_6$&$M_7$ \\ \hline
   $d_{t,6}$&1&2&3&4&3&2&1&2 \\ 
 \end{tabular}
 \end{equation}}

 \smallskip

We have $s(M_6)\le\frac{16}{48}$ by selecting paths which generate $g_{6,1}$, 
and we conclude $s(M_6)=\displaystyle\frac{16}{48}$ by Lemma\,\ref{key_lemma} and the following table. 

 {\footnotesize\[
 \begin{tabular}{c|cccccccc|c} 
   $M_t$&$R$&$M_1$&$M_2$&$M_3$&$M_4$&$M_5$&$M_6$&$M_7$&Total \\ \hline
   Path&$aBcEfG$&$BcEfG$&$cEfG$&$EfG$&$fG$&$G$&$1_{M_6}$&$dEfG$& \\ \hline
   $g_{6,1}$&1&2&3&4&3&2&1&0&16 \\
   $g_{6,2}$&0&2&3&4&3&2&1&2&17    
 \end{tabular}
 \]}

\bigskip

 \item[] \underline{\textbf{$\cdot$ the case of $M_7$ in $E_7$}\;}

 \medskip

Since $\rank_RM_7=2$, the multiplicity $d_{t,7}$ of $M_t$ in ${}^eM_7$ is the same as the table (\ref{E7_M1_multi}).
By selecting paths which generate $g_{7,1}$, we have $s(M_7)\le\frac{24}{48}$. 

 {\footnotesize\[
 \begin{tabular}{c|cccccccc|c} 
   $M_t$&$R$&$M_1$&$M_2$&$M_3$&$M_4$&$M_5$&$M_6$&$M_7$&Total \\ \hline
   Path&$aBcD$&$BcD$&$cD$&$D$&$eD$&$FeD$&0&$1_{M_7}$& \\ \hline
   $g_{7,1}$&2&4&6&8&0&0&0&4&24 \\
   $g_{7,2}$&0&0&6&8&6&4&0&4&28 \\
   $g_{7,3}$&0&4&6&8&6&4&0&4&32 \\  
   $g_{7,4}$&0&4&6&8&6&4&0&4&32   
 \end{tabular}
 \]}

Similarly, we conclude $s(M_7)=\displaystyle\frac{24}{48}$.

\end{itemize}

\subsection{Type $E_8$} 
\label{type_E8}

The AR quiver of type $E_8$ $($as the form of $\ZZ E_8)$ is 

{\footnotesize 
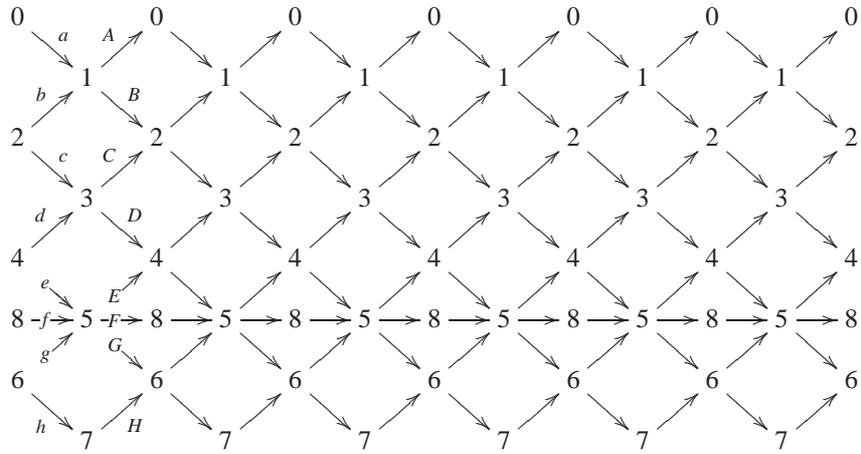
\[\xymatrix@C=15pt@R=10pt{
0\ar[rd]^a&&0\ar[rd]&&0\ar[rd]&&0\ar[rd]&&0\ar[rd]&&0\ar[rd]&&0 \\
&1\ar[ru]^A\ar[rd]^B&&1\ar[ru]\ar[rd]&&1\ar[ru]\ar[rd]&&1\ar[ru]\ar[rd]&&1\ar[ru]\ar[rd]&&1\ar[ru]\ar[rd]& \\
2\ar[ru]^b\ar[rd]^c&&2\ar[ru]\ar[rd]&&2\ar[ru]\ar[rd]&&2\ar[ru]\ar[rd]&&2\ar[ru]\ar[rd]&&2\ar[ru]\ar[rd]&&2 \\
&3\ar[ru]^C\ar[rd]^D&&3\ar[ru]\ar[rd]&&3\ar[ru]\ar[rd]&&3\ar[ru]\ar[rd]&&3\ar[ru]\ar[rd]&&3\ar[ru]\ar[rd]& \\
4\ar[ru]^d\ar[rd]|(0.4)e&&4\ar[ru]\ar[rd]&&4\ar[ru]\ar[rd]&&4\ar[ru]\ar[rd]&&4\ar[ru]\ar[rd]&&4\ar[ru]\ar[rd]&&4 \\
8\ar[r]|(0.4)f&5\ar[ru]|(0.4){\,E\,}\ar[r]|(0.4)F\ar[rd]|(0.4)G&8\ar[r]&5\ar[ru]\ar[r]\ar[rd]&8\ar[r]&5\ar[ru]\ar[r]\ar[rd]&8\ar[r]&5\ar[ru]\ar[r]\ar[rd]&8\ar[r]&
5\ar[ru]\ar[r]\ar[rd]&8\ar[r]&5\ar[ru]\ar[r]\ar[rd]&8 \\
6\ar[ru]|(0.4){\,g\,}\ar[rd]_h&&6\ar[ru]\ar[rd]&&6\ar[ru]\ar[rd]&&6\ar[ru]\ar[rd]&&6\ar[ru]\ar[rd]&&6\ar[ru]\ar[rd]&&6 \\
&7\ar[ru]_H&&7\ar[ru]&&7\ar[ru]&&7\ar[ru]&&7\ar[ru]&&7\ar[ru]& 
}\] }

with relations 
\[
\begin{cases}
aA=0,& bB+cC=0, \\
dD+eE=0,& fF=0, \\
gG+hH=0,& Aa+Bb=0, \\
Cc+Dd=0,& Ee+Ff+Gg=0, \\
Hh=0.
\end{cases}
\]
After applying the counting argument, we have the following quiver 

{\scriptsize 
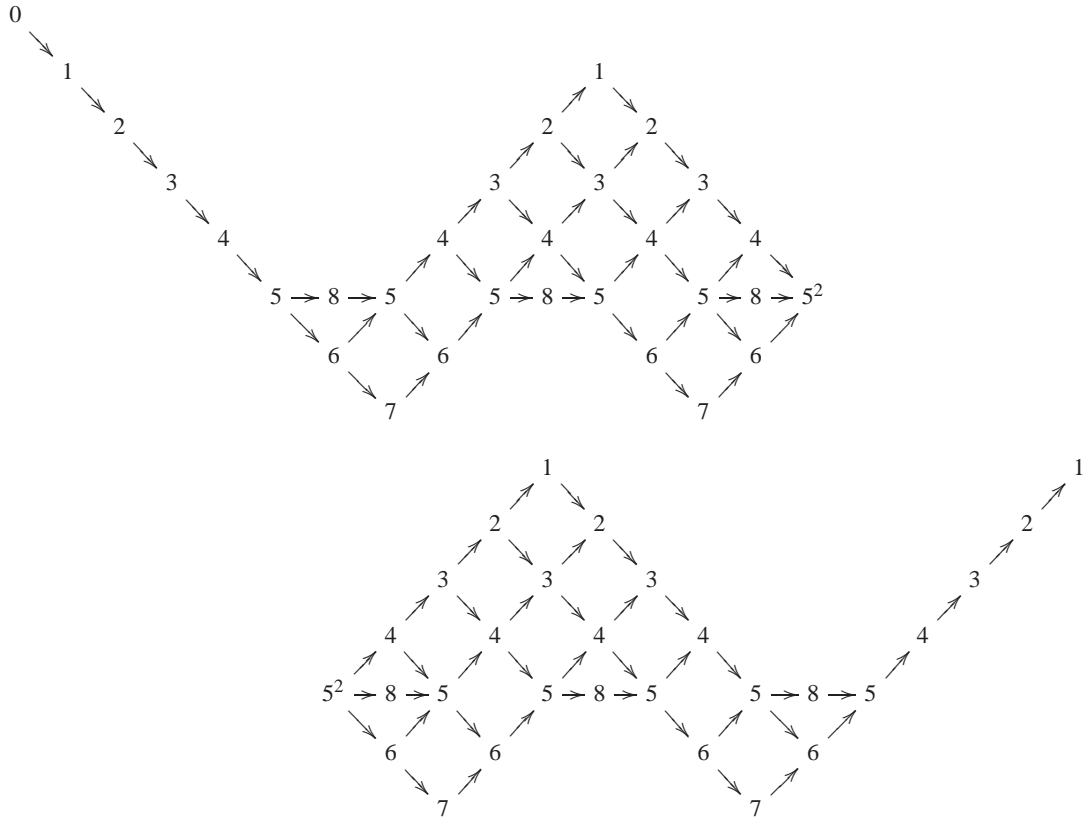
\[\xymatrix@C=8pt@R=8pt{
0\ar[rd]&&&&&&&&&&&&&&&& \\
&1\ar[rd]&&&&&&&&&&1\ar[rd]&&&&& \\
&&2\ar[rd]&&&&&&&&2\ar[ru]\ar[rd]&&2\ar[rd]&&&& \\
&&&3\ar[rd]&&&&&&3\ar[ru]\ar[rd]&&3\ar[ru]\ar[rd]&&3\ar[rd]&&& \\
&&&&4\ar[rd]&&&&4\ar[ru]\ar[rd]&&4\ar[ru]\ar[rd]&&4\ar[ru]\ar[rd]&&4\ar[rd]&& \\
&&&&&5\ar[r]\ar[rd]&8\ar[r]&5\ar[ru]\ar[rd]&&5\ar[ru]\ar[r]&8\ar[r]&5\ar[ru]\ar[rd]&&5\ar[ru]\ar[r]\ar[rd]&8\ar[r]&5^2& \\
&&&&&&6\ar[ru]\ar[rd]&&6\ar[ru]&&&&6\ar[ru]\ar[rd]&&6\ar[ru]&& \\
&&&&&&&7\ar[ru]&&&&&&7\ar[ru]&&& \\
&&&&&&&&&&1\ar[rd]&&&&&&&&&&1 \\
&&&&&&&&&2\ar[ru]\ar[rd]&&2\ar[rd]&&&&&&&&2\ar[ru]&& \\
&&&&&&&&3\ar[ru]\ar[rd]&&3\ar[ru]\ar[rd]&&3\ar[rd]&&&&&&3\ar[ru]&&& \\
&&&&&&&4\ar[ru]\ar[rd]&&4\ar[ru]\ar[rd]&&4\ar[ru]\ar[rd]&&4\ar[rd]&&&&4\ar[ru]&&&& \\
&&&&&&5^2\ar[ru]\ar[r]\ar[rd]&8\ar[r]&5\ar[ru]\ar[rd]&&5\ar[ru]\ar[r]&8\ar[r]&5\ar[ru]\ar[rd]&&5\ar[r]\ar[rd]&8\ar[r]&5\ar[ru]&&&&& \\
&&&&&&&6\ar[ru]\ar[rd]&&6\ar[ru]&&&&6\ar[ru]\ar[rd]&&6\ar[ru]&&&&&& \\
&&&&&&&&7\ar[ru]&&&&&&7\ar[ru]&&&&&&& 
}\] }

where the right side of upper part and the left side of lower part are identified. 

\begin{itemize}
 \item[] \underline{\textbf{$\cdot$ the case of $M_1$ in $E_8$}\;}

 \medskip

 Since $\rank_RM_1=2$, the multiplicity $d_{t,1}$ of $M_t$ in ${}^eM_1$ is the following.

{\footnotesize
 \begin{equation}
 \label{E8_M1_multi}
 \begin{tabular}{c|ccccccccc} 
   $M_t$&$R$&$M_1$&$M_2$&$M_3$&$M_4$&$M_5$&$M_6$&$M_7$&$M_8$ \\ \hline
   $d_{t,1}$&2&4&6&8&10&12&8&4&6 \\ 
 \end{tabular}
 \end{equation}}

 \smallskip

In a similar way to the other cases, we have the upper bounds $s(M_1)\le\frac{6}{120}$ by selecting paths which generate $g_{1,1}$, 
and the following table and Lemma\,\ref{key_lemma} assert the equality. 

 {\footnotesize\[
 \begin{tabular}{c|ccccccccc|c} 
   $M_t$&$R$&$M_1$&$M_2$&$M_3$&$M_4$&$M_5$&$M_6$&$M_7$&$M_8$&Total \\ \hline
   Path&$a$&$1_{M_1}$&$b$&0&0&0&0&0&0& \\ \hline
   $g_{1,1}$&2&4&0&0&0&0&0&0&0&6 \\
   $g_{1,2}$&0&4&6&0&0&0&0&0&0&10 \\  
   $g_{1,3}$&0&4&6&0&0&0&0&0&0&10 \\ 
   $g_{1,4}$&0&4&6&0&0&0&0&0&0&10 \\ 
 \end{tabular}
 \]}

\bigskip

 \item[] \underline{\textbf{$\cdot$ the case of $M_2$ in $E_8$}\;}

 \medskip

 Since $\rank_RM_2=3$, the multiplicity $d_{t,2}$ of $M_t$ in ${}^eM_2$ is the following.

{\footnotesize
 \begin{equation}
 \label{E8_M2_multi}
 \begin{tabular}{c|ccccccccc} 
   $M_t$&$R$&$M_1$&$M_2$&$M_3$&$M_4$&$M_5$&$M_6$&$M_7$&$M_8$ \\ \hline
   $d_{t,2}$&3&6&9&12&15&18&12&6&9 \\ 
 \end{tabular}
 \end{equation}}

 \smallskip

We have the upper bounds $s(M_2)\le\frac{18}{120}$ by selecting paths which generate $g_{2,1}$. 
Thus, we conclude $s(M_2)=\displaystyle\frac{18}{120}$ by the following table and Lemma\,\ref{key_lemma}. 

 {\footnotesize\[
 \begin{tabular}{c|ccccccccc|c} 
   $M_t$&$R$&$M_1$&$M_2$&$M_3$&$M_4$&$M_5$&$M_6$&$M_7$&$M_8$&Total \\ \hline
   Path&$aB$&$B$&$1_{M_2}$&$C$&0&0&0&0&0& \\ \hline
   $g_{2,1}$&3&6&9&0&0&0&0&0&0&18 \\
   $g_{2,2}$&0&0&9&12&0&0&0&0&0&21 \\  
   $g_{2,3}$&0&6&9&12&0&0&0&0&0&27 \\ 
   $g_{2,4}$&0&0&9&12&0&0&0&0&0&21 \\ 
   $g_{2,5}$&0&6&9&12&0&0&0&0&0&27 \\ 
   $g_{2,6}$&0&0&9&12&0&0&0&0&0&21 
 \end{tabular}
 \]}

\bigskip

 \item[] \underline{\textbf{$\cdot$ the case of $M_3$ in $E_8$}\;}

 \medskip

 Since $\rank_RM_3=4$, the multiplicity $d_{t,3}$ of $M_t$ in ${}^eM_3$ is the following.

{\footnotesize
 \begin{equation}
 \label{E8_M3_multi}
 \begin{tabular}{c|ccccccccc} 
   $M_t$&$R$&$M_1$&$M_2$&$M_3$&$M_4$&$M_5$&$M_6$&$M_7$&$M_8$ \\ \hline
   $d_{t,3}$&4&8&12&16&20&24&16&8&12 \\ 
 \end{tabular}
 \end{equation}}

 \smallskip

We have the upper bounds $s(M_3)\le\frac{40}{120}$ by selecting paths which generate $g_{3,1}$. 
Thus, we conclude $s(M_3)=\displaystyle\frac{40}{120}$ by the following table and Lemma\,\ref{key_lemma}. 

 {\footnotesize\[
 \begin{tabular}{c|ccccccccc|c} 
   $M_t$&$R$&$M_1$&$M_2$&$M_3$&$M_4$&$M_5$&$M_6$&$M_7$&$M_8$&Total \\ \hline
   Path&$aBc$&$Bc$&$c$&$1_{M_3}$&0&$Ed$&0&0&0& \\ \hline
   $g_{3,1}$&4&8&12&16&0&0&0&0&0&40 \\
   $g_{3,2}$&0&0&0&16&0&24&0&0&0&40 \\  
   $g_{3,3}$&0&0&12&16&0&24&0&0&0&52 \\ 
   $g_{3,4}$&0&8&12&16&0&24&0&0&0&60 \\ 
   $g_{3,5}$&0&0&0&16&0&24&0&0&0&40 \\ 
   $g_{3,6}$&0&0&12&16&0&24&0&0&0&52 \\
   $g_{3,7}$&0&8&12&16&0&24&0&0&0&60 \\ 
   $g_{3,8}$&0&0&0&16&0&24&0&0&0&40 
 \end{tabular}
 \]}

\bigskip

 \item[] \underline{\textbf{$\cdot$ the case of $M_4$ in $E_8$}\;}

 \medskip

 Since $\rank_RM_4=5$, the multiplicity $d_{t,4}$ of $M_t$ in ${}^eM_4$ is the following.

{\footnotesize
 \begin{equation}
 \begin{tabular}{c|ccccccccc} 
   $M_t$&$R$&$M_1$&$M_2$&$M_3$&$M_4$&$M_5$&$M_6$&$M_7$&$M_8$ \\ \hline
   $d_{t,4}$&5&10&15&20&25&30&20&10&15 \\ 
 \end{tabular}
 \end{equation}}

 \smallskip

Similarly, we have $s(M_4)\le\frac{75}{120}$. 
The following table and Lemma\,\ref{key_lemma} assert the equality. 

 {\footnotesize\[
 \begin{tabular}{c|ccccccccc|c} 
   $M_t$&$R$&$M_1$&$M_2$&$M_3$&$M_4$&$M_5$&$M_6$&$M_7$&$M_8$&Total \\ \hline
   Path&$aBcD$&$BcD$&$cD$&$D$&$1_{M_4}$&$E$&$gE$&0&0& \\ \hline
   $g_{4,1}$&5&10&15&20&25&0&0&0&0&75 \\
   $g_{4,2}$&0&0&0&0&25&30&20&0&0&75 \\  
   $g_{4,3}$&0&0&0&20&25&30&20&0&0&95 \\ 
   $g_{4,4}$&0&0&15&20&25&30&0&0&0&90 \\ 
   $g_{4,5}$&0&10&15&20&25&30&20&0&0&120 \\ 
   $g_{4,6}$&0&0&0&0&25&30&20&0&0&75 \\
   $g_{4,7}$&0&0&0&20&25&30&20&0&0&95 \\ 
   $g_{4,8}$&0&0&15&20&25&30&20&0&0&110 \\ 
   $g_{4,9}$&0&10&15&20&25&30&0&0&0&100 \\ 
   $g_{4,10}$&0&0&0&0&25&30&20&0&0&75 
 \end{tabular}
 \]}

\bigskip

 \item[] \underline{\textbf{$\cdot$ the case of $M_5$ in $E_8$}\;}

 Since $\rank_RM_5=6$, the multiplicity $d_{t,5}$ of $M_t$ in ${}^eM_5$ is the following.

{\footnotesize
 \begin{equation}
 \begin{tabular}{c|ccccccccc} 
   $M_t$&$R$&$M_1$&$M_2$&$M_3$&$M_4$&$M_5$&$M_6$&$M_7$&$M_8$ \\ \hline
   $d_{t,5}$&6&12&18&24&30&36&24&12&18 \\ 
 \end{tabular}
 \end{equation}}

 \smallskip

In the same way as before, we can classify the MCM $R$-modules appear in ${}^eM_5$ as 
\[
(\rmI)\;\{\;M_5\times 36\;\}\quad(\rmI\rmI)\;\{\;R\times 6, M_1\times 12, M_2\times 18, M_3\times 24, M_4\times 30\;\}\quad
(\rmI\rmI\rmI)\;\{\;M_6\times 24, M_8\times 18\;\}
\]
where the class$(\rmI)$ $($resp. $(\rmI\rmI),(\rmI\rmI\rmI))$ is the set of MCM $R$-modules which generate $g_{5,1}$ and $g_{5,2}$ at the same time 
$($resp. either $g_{5,1}$ or $g_{5,2}$, only $g_{5,2}$ $)$ and obtain the upper bound $s(M_5)\le\frac{36}{120}+\frac{42}{120}+\frac{24}{120}=\frac{102}{120}$. 
By Lemma\,\ref{key_lemma} and the following table, we have $s(M_5)=\displaystyle\frac{102}{120}$. 
In this table, we fix that $g_{5,6}\,($resp. $g_{5,7})$ is a minimal generator identified with a path 
which factor through $4\xrightarrow{e}5^2\,($resp. $8\xrightarrow{f}5^2)$, 
and we can use paths which factor through $6\xrightarrow{g}5^2$ for generating either $g_{5,6}$ or $g_{5,7}$. 

 {\footnotesize\[
 \begin{tabular}{c|ccccccccc|c} 
   $M_t$&$R$&$M_1$&$M_2$&$M_3$&$M_4$&$M_5$&$M_6$&$M_7$&$M_8$&Total \\ \hline
   Path&$aBcDeFf$&$BcDe$&$cDeFf$&$De$&$e$&$1_{M_5}$&$g$&$Hg$&$f$& \\ \hline
   $g_{5,1}$&0&12&0&24&30&36&0&0&0&102 \\
   $g_{5,2}$&6&0&18&0&0&36&24&0&18&102 \\  
   $g_{5,3}$&0&0&0&0&30&36&24&12&0&102 \\ 
   $g_{5,4}$&0&0&0&24&30&36&0&0&18&108 \\ 
   $g_{5,5}$&0&0&0&24&20&36&24&0&0&114 \\ 
   $g_{5,6}$&0&12&0&24&30&36&0&0&0&102 \\
   $g_{5,7}$&0&0&18&0&0&36&24&12&18&108 \\ 
   $g_{5,8}$&0&0&18&0&30&36&24&0&18&126 \\ 
   $g_{5,9}$&0&0&0&24&30&36&24&12&0&126 \\ 
   $g_{5,10}$&0&0&0&24&30&36&0&0&18&108 \\ 
   $g_{5,11}$&0&12&18&24&30&36&24&0&0&144 \\ 
   $g_{5,12}$&0&0&18&0&0&36&24&12&18&108 
 \end{tabular}
 \]}

\bigskip

 \item[] \underline{\textbf{$\cdot$ the case of $M_6$ in $E_8$}\;}

 \medskip

 Since $\rank_RM_6=4$, the multiplicity $d_{t,6}$ of $M_t$ in ${}^eM_6$ is the same as the table (\ref{E8_M3_multi}). 
In the same way as before, we can classify the MCM $R$-modules appear in ${}^eM_6$ as 
\[
(\rmI)\;\{\;M_5\times 24, M_6\times 16\;\}\quad(\rmI\rmI)\;\{\;R\times 4, M_1\times 8, M_2\times 12, M_3\times 16, M_4\times 20\;\}\quad
(\rmI\rmI\rmI)\;\{\;M_7\times 8, M_8\times 12\;\}
\]
where the class$(\rmI)$ $($resp. $(\rmI\rmI),(\rmI\rmI\rmI))$ is the set of MCM $R$-modules which generate $g_{6,1}$ and $g_{6,2}$ at the same time 
$($resp. either $g_{6,1}$ or $g_{6,2}$, only $g_{6,2}$ $)$ and obtain the upper bound $s(M_6)\le\frac{40}{120}+\frac{20}{120}+\frac{20}{120}=\frac{80}{120}$. 
The following table and Lemma\,\ref{key_lemma} assert the equality. 

 {\footnotesize\[
 \begin{tabular}{c|cccccccccc|c} 
   $M_t$&$R$&$M_1$&$M_2$&$M_3$&$M_4$&$M_4$&$M_5$&$M_6$&$M_7$&$M_8$&Total \\ \hline
   Path&$aBcDeFfG$&$BcDeFfG$&$cDeG$&$DeG$&$eG$&$eFfG$&$G$&$1_{M_6}$&$H$&$fG$& \\ \hline
   $g_{6,1}$&0&0&12&16&12&0&24&16&0&0&80 \\
   $g_{6,2}$&4&8&0&0&0&8&24&16&8&12&80 \\  
   $g_{6,3}$&0&0&0&16&12&8&24&16&0&12&88 \\ 
   $g_{6,4}$&0&0&12&16&12&0&24&16&8&0&88 \\ 
   $g_{6,5}$&0&0&12&16&12&8&24&16&0&12&100 \\ 
   $g_{6,6}$&0&8&0&0&12&8&24&16&8&12&88 \\
   $g_{6,7}$&0&0&12&16&12&8&24&16&0&12&100 \\ 
   $g_{6,8}$&0&0&12&16&12&0&24&16&8&0&88 
 \end{tabular}
 \]}

\bigskip

 \item[] \underline{\textbf{$\cdot$ the case of $M_7$ in $E_8$}\;}

 \medskip

 Since $\rank_RM_7=2$, the multiplicity $d_{t,7}$ of $M_t$ in ${}^eM_7$ is the same as the table (\ref{E8_M1_multi}).
In the same way as before, we can classify the MCM $R$-modules appear in ${}^eM_7$ as 
\[
(\rmI)\;\{\;M_3\times 8, M_4\times 10, M_5\times 12, M_6\times 8, M_7\times 4\;\}\quad(\rmI\rmI)\;\{\;R\times 2, M_1\times 4, M_2\times 6\;\}\quad
(\rmI\rmI\rmI)\;\{\;M_8\times 6\;\}
\]
where the class$(\rmI)$ $($resp. $(\rmI\rmI),(\rmI\rmI\rmI))$ is the set of MCM $R$-modules which generate $g_{7,1}$ and $g_{7,2}$ at the same time 
$($resp. either $g_{7,1}$ or $g_{7,2}$, only $g_{7,2}$ $)$ and obtain the upper bound $s(M_7)\le\frac{42}{120}+\frac{6}{120}+\frac{3}{120}=\frac{51}{120}$. 
The following table and Lemma\,\ref{key_lemma} assert the equality. 

 {\footnotesize\[
 \begin{tabular}{c|cccccccccc|c} 
   $M_t$&$R$&$M_1$&$M_2$&$M_2$&$M_3$&$M_4$&$M_5$&$M_6$&$M_7$&$M_8$&Total \\ \hline
   Path&$aBcDeGh$&$BcDeGh$&$cDeGh$&$cDeF\cdots Gh$&$DeGh$&$eGh$&$Gh$&$h$&$1_{M_7}$&$fGh$& \\ \hline
   $g_{7,1}$&2&4&3&0&8&10&12&8&4&0&51 \\
   $g_{7,2}$&0&0&0&3&8&10&12&8&4&6&51 \\  
   $g_{7,3}$&0&4&3&0&8&10&12&8&4&6&55 \\ 
   $g_{7,4}$&0&0&3&3&8&10&12&8&4&6&54 \\ 
 \end{tabular}
 \]}

\bigskip

 \item[] \underline{\textbf{$\cdot$ the case of $M_8$ in $E_8$}\;}

 \medskip

 Since $\rank_RM_8=3$, the multiplicity $d_{t,8}$ of $M_t$ in ${}^eM_8$ is the same as the table (\ref{E8_M2_multi}).
In the same way as before, we can classify the MCM $R$-modules appear in ${}^eM_8$ as 
\[
(\rmI)\;\{\;M_4\times 15, M_5\times 18, M_8\times 9\;\}\quad(\rmI\rmI)\;\{\;R\times 3, M_1\times 6, M_2\times 9, M_3\times 12\;\}\quad
(\rmI\rmI\rmI)\;\{\;M_6\times 12, M_7\times 6\;\}
\]
where the class$(\rmI)$ $($resp. $(\rmI\rmI),(\rmI\rmI\rmI))$ is the set of MCM $R$-modules which generate $g_{8,1}$ and $g_{8,2}$ at the same time 
$($resp. either $g_{8,1}$ or $g_{8,2}$, only $g_{8,2}$ $)$ and obtain the upper bound $s(M_8)\le\frac{42}{120}+\frac{18}{120}+\frac{6}{120}=\frac{66}{120}$. 
The following table and Lemma\,\ref{key_lemma} assert the equality. 

 {\footnotesize\[
 \begin{tabular}{c|ccccccccc|c} 
   $M_t$&$R$&$M_1$&$M_2$&$M_3$&$M_4$&$M_5$&$M_6$&$M_7$&$M_8$&Total \\ \hline
   Path&$aBcDeF$&$aBcDeFfEeF$&$cDeF$&$DeF$&$eF$&$F$&$gF$&$HgF$&$1_{M_8}$& \\ \hline
   $g_{8,1}$&3&0&9&12&15&18&0&0&9&66 \\
   $g_{8,2}$&0&6&0&0&15&18&12&6&9&66 \\  
   $g_{8,3}$&0&0&9&12&15&18&12&0&9&75 \\ 
   $g_{8,4}$&0&0&9&12&15&18&12&6&9&81 \\ 
   $g_{8,5}$&0&6&0&12&15&18&12&6&9&78 \\ 
   $g_{8,6}$&0&0&9&12&15&18&12&0&9&75 
 \end{tabular}
 \]}

\end{itemize}

\section{Summary of the value of the dual $F$-signature}
\label{summary}

\begin{thm} 
\label{main_thm}
The following is the Dynkin diagram $Q$ and corresponding values of the dual $F$-signature 
$($In order to show the ratio of dual $F$-signature to the order of $G$ clearly, we don't reduce fractions $)$. 

\begin{itemize}
\item [(1)] Type $A_n$

 \begin{itemize}
   \item [$\bullet$] $n$ is an even number (i.e. $n=2r$)
 {\scriptsize
 \[\xymatrix@C=10pt@R=7pt{
   A_n:&1\ar@{-}[r]&2\ar@{-}[r]&\cdots\ar@{-}[r]&r\ar@{-}[r]&r+1\ar@{-}[r]&\cdots\ar@{-}[r]&n-1\ar@{-}[r]&n \\
   &&&&&&&&\\
   &\displaystyle\frac{2}{n+1}\ar@{-}[r]&\displaystyle\frac{3}{n+1}\ar@{-}[r]&\cdots\ar@{-}[r]&\displaystyle\frac{r+1}{n+1}\ar@{-}[r]&\displaystyle\frac{r+1}{n+1}\ar@{-}[r]&\cdots\ar@{-}[r]&\displaystyle\frac{3}{n+1}\ar@{-}[r]&\displaystyle\frac{2}{n+1}
 }\] }

  \item [$\bullet$] $n$ is an odd number (i.e. $n=2r-1$)
 {\scriptsize 
 \[\xymatrix@C=10pt@R=7pt{
   A_n:&1\ar@{-}[r]&2\ar@{-}[r]&\cdots\ar@{-}[r]&r-1\ar@{-}[r]&r\ar@{-}[r]&r+1\ar@{-}[r]&\cdots\ar@{-}[r]&n-1\ar@{-}[r]&n \\
   &&&&&&&&&\\
   &\displaystyle\frac{2}{n+1}\ar@{-}[r]&\displaystyle\frac{3}{n+1}\ar@{-}[r]&\cdots\ar@{-}[r]&\displaystyle\frac{r}{n+1}\ar@{-}[r]&\displaystyle\frac{2r+1}{2(n+1)}\ar@{-}[r]&\displaystyle\frac{r}{n+1}\ar@{-}[r]&\cdots\ar@{-}[r]&\displaystyle\frac{3}{n+1}\ar@{-}[r]&\displaystyle\frac{2}{n+1}
 }\] }
 \end{itemize}

\item [(2)] Type $D_n$

 \begin{itemize}
   \item [$\bullet$] $n$ is an even number (i.e. $n=2r$)
 {\scriptsize 
 \[\xymatrix@C=7pt@R=5pt{
   &&&&&&&&&&n-1 \\
   D_n:&1\ar@{-}[r]&2\ar@{-}[r]&\cdots\ar@{-}[r]&m\ar@{-}[r]&\cdots\ar@{-}[r]&r\ar@{-}[r]&r+1\ar@{-}[r]&\cdots\ar@{-}[r]&n-2\ar@{-}[ur]\ar@{-}[dr]& \\ 
   &&&&&&&&&&n \\
   &&&&&&&&&& \\
   &&&&&&&&&&\displaystyle\frac{2(n-2)}{4(n-2)} \\
   &\displaystyle\frac{4}{4(n-2)}\ar@{-}[r]&\displaystyle\frac{6}{4(n-2)}\ar@{-}[r]&\cdots\ar@{-}[r]&\displaystyle\frac{4m-2}{4(n-2)}\ar@{-}[r]&\cdots\ar@{-}[r]&
\displaystyle\frac{4r-2}{4(n-2)}\ar@{-}[r]&\displaystyle\frac{4r-1}{4(n-2)}\ar@{-}[r]&\cdots\ar@{-}[r]&\displaystyle\frac{4r-1}{4(n-2)}\ar@{-}[ur]\ar@{-}[dr]& \\ 
   &&&&&&&&&&\displaystyle\frac{2(n-2)}{4(n-2)} \\
 }\] }

  \item [$\bullet$] $n$ is an odd number (i.e. $n=2r-1$)
 {\scriptsize 
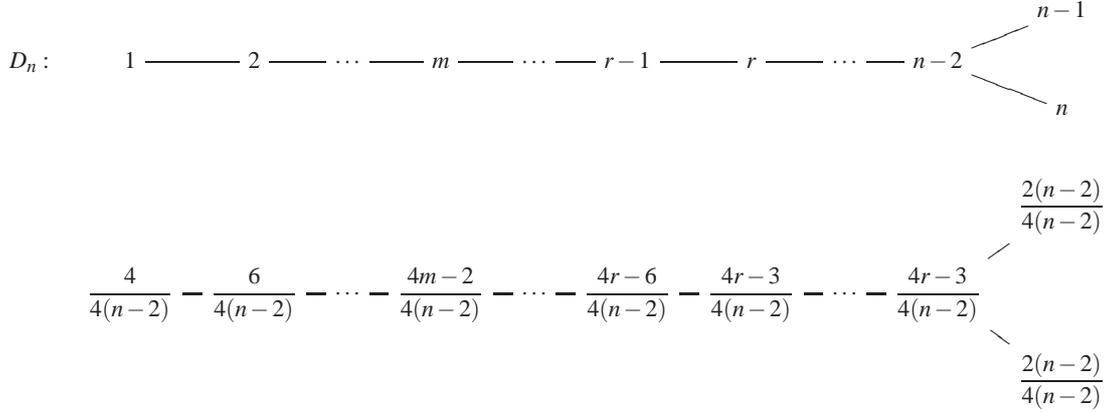
 \[\xymatrix@C=7pt@R=5pt{
   &&&&&&&&&&n-1 \\
   D_n:&1\ar@{-}[r]&2\ar@{-}[r]&\cdots\ar@{-}[r]&m\ar@{-}[r]&\cdots\ar@{-}[r]&r-1\ar@{-}[r]&r\ar@{-}[r]&\cdots\ar@{-}[r]&n-2\ar@{-}[ur]\ar@{-}[dr]& \\ 
   &&&&&&&&&&n \\
   &&&&&&&&&& \\
   &&&&&&&&&&\displaystyle\frac{2(n-2)}{4(n-2)} \\
   &\displaystyle\frac{4}{4(n-2)}\ar@{-}[r]&\displaystyle\frac{6}{4(n-2)}\ar@{-}[r]&\cdots\ar@{-}[r]&\displaystyle\frac{4m-2}{4(n-2)}\ar@{-}[r]&\cdots\ar@{-}[r]&
\displaystyle\frac{4r-6}{4(n-2)}\ar@{-}[r]&\displaystyle\frac{4r-3}{4(n-2)}\ar@{-}[r]&\cdots\ar@{-}[r]&\displaystyle\frac{4r-3}{4(n-2)}\ar@{-}[ur]\ar@{-}[dr]& \\ 
   &&&&&&&&&&\displaystyle\frac{2(n-2)}{4(n-2)} \\
 }\] }
 \end{itemize}

\item [(3)] Type $E_6$
 {\scriptsize 
 
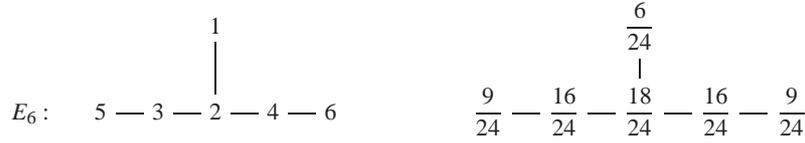
\[\xymatrix@C=10pt@R=7pt{
   &&&1&& && &&&\displaystyle\frac{6}{24}&& \\
   E_6:&5\ar@{-}[r]&3\ar@{-}[r]&2\ar@{-}[r]\ar@{-}[u]&4\ar@{-}[r]&6 && 
   &\displaystyle\frac{9}{24}\ar@{-}[r]&\displaystyle\frac{16}{24}\ar@{-}[r]&\displaystyle\frac{18}{24}\ar@{-}[r]\ar@{-}[u]&\displaystyle\frac{16}{24}\ar@{-}[r]&\displaystyle\frac{9}{24}
 }\] }

\item [(4)] Type $E_7$
 {\scriptsize 
 
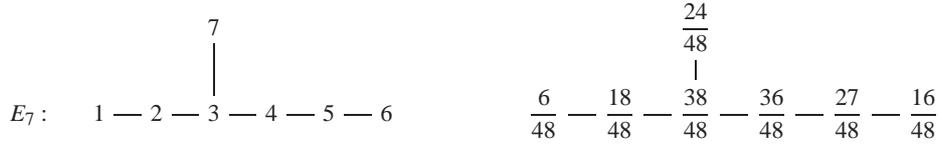
\[\xymatrix@C=10pt@R=7pt{
   &&&7&&& && &&&\displaystyle\frac{24}{48}&&& \\
   E_7:&1\ar@{-}[r]&2\ar@{-}[r]&3\ar@{-}[r]\ar@{-}[u]&4\ar@{-}[r]&5\ar@{-}[r]&6 &&    
   &\displaystyle\frac{6}{48}\ar@{-}[r]&\displaystyle\frac{18}{48}\ar@{-}[r]&\displaystyle\frac{38}{48}\ar@{-}[r]\ar@{-}[u]&\displaystyle\frac{36}{48}\ar@{-}[r]&\displaystyle\frac{27}{48}\ar@{-}[r]&\displaystyle\frac{16}{48}
 }\] }

\item [(5)] Type $E_8$
 {\scriptsize 
 
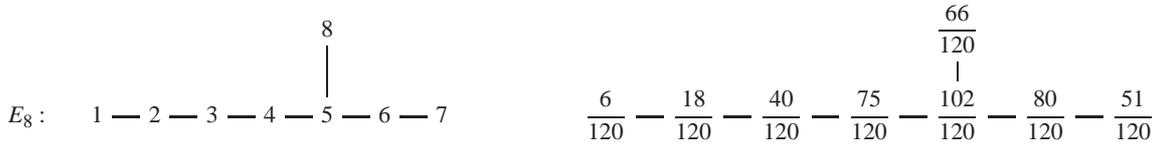
\[\xymatrix@C=10pt@R=7pt{
   &&&&&8&& && &&&&&\displaystyle\frac{66}{120}&& \\
   E_8:&1\ar@{-}[r]&2\ar@{-}[r]&3\ar@{-}[r]&4\ar@{-}[r]&5\ar@{-}[r]\ar@{-}[u]&6\ar@{-}[r]&7 &&    
   &\displaystyle\frac{6}{120}\ar@{-}[r]&\displaystyle\frac{18}{120}\ar@{-}[r]&\displaystyle\frac{40}{120}\ar@{-}[r]&\displaystyle\frac{75}{120}\ar@{-}[r]&\displaystyle\frac{102}{120}\ar@{-}[r]\ar@{-}[u]&\displaystyle\frac{80}{120}\ar@{-}[r]&\displaystyle\frac{51}{120}
 }\] }

\end{itemize}
\end{thm}

\begin{rem}
As these lists show, we have $s(M_t)=s(M_t^*)$. Indeed, each AR quiver is symmetric with respect to $M_t$ and $M_t^*$, 
and $\rank_RM_t=\rank_RM_t^*$. Thus, it follows from arguments used in Section\,\ref{DFsig_RDP}.
\end{rem}


\appendix
\section{Hilbert-Kunz multiplicity of quotient surface singularities}

By using arguments similar to those in Section\,\ref{Gene_Fsig} and \ref{tau_cat}, we can investigate the Hilbert-Kunz multiplicity. 
The study of this numerical invariant in positive characteristic was started in \cite{Kun2} and its existence was shown by P.~Monsky \cite{Mon}.

\begin{df-th}[Hilbert-Kunz multiplicity] 
Let $(R,\fkm,k)$ be a Noetherian local ring of characteristic $p>0$ and $I$ be an $\fkm$-primary ideal of $R$. 
Then the limit 
\[
e_{\text{HK}}(I;R)\coloneqq\lim_{e\rightarrow\infty}\frac{1}{p^{ed}}l_R\big(R/I^{[p^e]}\big) 
\]
exists \cite{Mon}, where $l_R(-)$ stands for the length of a finitely generated Artinian $R$-module and $I^{[p^e]}$ is 
the ideal generated by the $p^e$-th powers of the element of $I$. 
We call this limit the Hilbert-Kunz multiplicity of $R$ with respect to $I$. 
In particular, $e_{\text{HK}}(\fkm;R)\coloneqq e_{\text{HK}}(R)$ is called the Hilbert-Kunz multiplicity of $R$.
\end{df-th}

This invariant plays an important role to investigate singularities of positive characteristic. 
For example if $R$ is regular, then $e_{\text{HK}}(R)=1$. The converse holds if $R$ is unmixed \cite{WY1}, \cite{HY}. 
Therefore the properties and its value were observed in many articles $($for more details, see the survey article \cite{Hun} and the references contained therein$)$.  
However, in the spite of its importance, it is difficult to determine the explicit value of $e_{\text{HK}}(R)$ in general. 
For quotient singularities, that is, the case we are interested in this paper, the Hilbert-Kunz multiplicity is determined by the next formula.

\begin{thm}$($cf. \cite[Theorem\,2.7]{WY1}$)$
Let the notation be same as Section\,\ref{Gene_Fsig}. Then 
\[
e_{\text{HK}}(R)=\frac{1}{|G|}l_S\big(S/\fkm S\big).
\]
\end{thm}

\bigskip

Then we deform it as follows. 
Since $S\cong R^{\oplus d_0}\oplus M_1^{\oplus d_1}\oplus\cdots\oplus M_n^{\oplus d_n}$\;$(d_t=\rank_RM_t=\dim_kV_t)$, 
\[
l_s(S/\fkm S)=\dim_k(S\otimes R/\fkm)=\mu_R(S)=\sum^n_{t=0}d_t\;\mu_R(M_t) ,
\]
where $\mu_R(M)$ stands for the number of minimal generator of a finitely generated $R$-module $M$. Thus, 
\[
e_{\text{HK}}(R)=\frac{1}{|G|}\sum^n_{t=0}d_t\;\mu_R(M_t) .
\]
Note that this formula is also obtained by the isomorphism (\ref{iso_p^ed}) and Theorem\,\ref{gene-Fsig}.
As we showed in subsection\,\ref{count_AR}, we can calculate $\mu_R(M_t)$ by using the AR quiver (or the McKay quiver). 
Thus, we can determine the value of the Hilbert-Kunz multiplicity by a relatively easy process. 

\begin{ex}$($\cite[Theorem\,5.4]{WY1}, see also \cite[Corollary\,20]{HL}, \cite[Corollary\,4.15]{Tuc}$)$
In the same as Section\,\ref{DFsig_RDP}, let $R$ be a two-dimensional rational double point. 
For an indecomposable $R$-module $M_t$, we have $\mu_R(M_t)=2d_t$ for $t\neq 0$ $($see Section\,\ref{DFsig_RDP}$)$ and clearly $\mu_R(R)=d_0=1$. 
Thus, we have 
\[
e_{\text{HK}}(R)=\frac{1}{|G|}(2\sum^n_{t=0}d_t^2-1)=\frac{1}{|G|}(2|G|-1)=2-\frac{1}{|G|}.
\]
\end{ex}

\medskip

\begin{ex}
Let $G\coloneqq \langle\;\sigma=
    \begin{pmatrix} \zeta_8&0 \\
                    0&\zeta_8^5 
    \end{pmatrix}           \;\rangle$ 
be a cyclic group of order $8$ where $\zeta_8$ is a primitive $8$-th root of unity. 
We consider irreducible representations of $G$; 
\[
V_t:\sigma\mapsto\zeta_8^{-t} \quad(t=0,1,\cdots,7). 
\]
In the same way as Section\,\ref{Gene_Fsig}, we set the invariant subring $R\coloneqq S^G$ and 
its indecomposable MCM module $M_t\coloneqq (S\otimes_kV_t)^G$. 
By the counting argument of AR quiver, we obtain the following 
(the meaning of this picture, see Section\,\ref{DFsig_RDP}). 

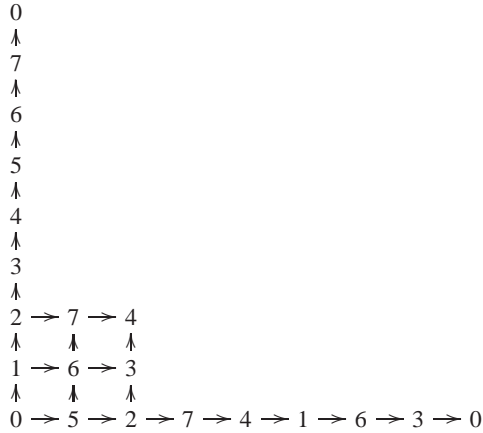
\[
 \begin{array}{cc}
 \begin{array}{c}
{\scriptsize 
\xymatrix@C=10pt@R=6pt{
0 \\
7\ar[u] \\
6\ar[u] \\
5\ar[u] \\
4\ar[u] \\
3\ar[u] \\
2\ar[u]\ar[r]&7\ar[r]&4 \\
1\ar[u]\ar[r]&\ar[u]6\ar[r]&3\ar[u] \\
0\ar[u]\ar[r]&5\ar[u]\ar[r]&2\ar[u]\ar[r]&7\ar[r]&4\ar[r]&1\ar[r]&6\ar[r]&3\ar[r]&0
} }
 \end{array}&
 \begin{array}{c}
 {\small
 \begin{tabular}{c|cccccccc} 
   $t$&$0$&$1$&$2$&$3$&$4$&$5$&$6$&$7$ \\ \hline
   $\mu(M_t)$&$1$&$2$&$2$&$3$&$3$&$2$&$3$&$3$ \\ \hline
   $\rank M_t$&$1$&$1$&$1$&$1$&$1$&$1$&$1$&$1$ 
 \end{tabular} }
 \end{array} 
 \end{array}
\]
So we have $e_{\text{HK}}(R)=\displaystyle\frac{19}{8}$. 
\end{ex}

\medskip

\begin{ex}
Let the notation be same as Example\,\ref{D_52_group}. 
By the counting argument of AR quiver, we have the number of minimal generators as follows. 
 {\footnotesize\[
 \begin{tabular}{c|cccccccc} 
   $(i,j)$&$(0,0)$&$(1,0)$&$(3,0)$&$(4,0)$&$(2,2)$&$(0,1)$&$(1,1)$&$(3,1)$ \\ \hline
   $\mu(M_{i,j})$&$1$&$3$&$3$&$3$&$4$&$3$&$2$&$2$ \\ \hline
   $\rank M_{i,j}$&$1$&$1$&$1$&$1$&$2$&$1$&$1$&$1$ 
 \end{tabular}\]

 \[\begin{tabular}{c|ccccccc} 
   &$(4,1)$&$(2,0)$&$(0,2)$&$(1,2)$&$(3,2)$&$(4,2)$&$(2,1)$ \\ \hline
   &$2$&$5$&$2$&$3$&$3$&$3$&$6$ \\ \hline
   &$1$&$2$&$1$&$1$&$1$&$1$&$2$ 
 \end{tabular}
 \]}

Thus, we have $e_{\text{HK}}(R)=\displaystyle\frac{60}{24}=\frac{5}{2}$.

\end{ex}

\subsection*{Acknowledgements}
The author is grateful to Professor Mitsuyasu Hashimoto for giving him valuable advice and encouragements.  
The author also thanks Professor Osamu Iyama for explaining him about counting arguments of Auslander-Reiten quiver 
and thanks Professor Kei-ichi Watanabe, Professor Yuji Yoshino, Professor Yukari Ito and Akiyoshi Sannai for valuable comments and conversations. 

The author is partially supported by Grant-in-Aid for JSPS Fellows $($No. 26-422$)$.


\end{document}